\documentclass[journal,onecolumn,12pt,twoside]{IEEEtran}
\usepackage[numbers,sort]{natbib}

\usepackage[cmex10]{amsmath}
\usepackage{mathtools}
\allowdisplaybreaks[4]

\usepackage{array}
\usepackage{mdwmath}
\usepackage{mdwtab}

\usepackage{url}
\usepackage[normalem]{ulem}
	
\usepackage{graphicx} 
\usepackage{epsfig} 
\usepackage{epstopdf}
\usepackage{amssymb}  
\usepackage{verbatim}
\usepackage[usenames,dvipsnames,svgnames]{xcolor}
\usepackage{multirow} 
\usepackage{dsfont}
\graphicspath{ {./images/} }
\usepackage{verbatim}
\usepackage{subcaption}
\usepackage{wrapfig}
\newcolumntype{L}{>{$}l<{$}}
\usepackage[inline,shortlabels]{enumitem}
\setlist[itemize]{leftmargin=*}
\setlist[enumerate]{leftmargin=*}

\def\lpr{\left\{}
\def\rpr{\right\}}
	
\newtheorem{claim}{Claim}
\newtheorem{theorem}{Theorem}
\newtheorem{definition}{Definition}
\newtheorem{lemma}{Lemma}

\DeclareMathOperator*{\argmax}{argmax}
\newcommand{\mA}{\mathcal{A}}
\newcommand{\mX}{\mathcal{X}}

\newcommand{\mN}{\mathcal{N}}
\newcommand{\upi}{\underline{\pi}}
\newcommand{\uPi}{\underline{\Pi}}
\renewcommand{\hat}{\widehat}
\renewcommand{\bar}{\overline}
\newcommand{\ux}{\underline{x}}
\renewcommand{\star}{\ast}

\newcommand{\uF}{\underline{F}}
\newcommand{\mP}{\mathbb{P}}
\newcommand{\mE}{\mathbb{E}}	

\newcommand{\eq}[1]{\begin{align}#1\end{align}}
\newcommand{\seq}[1]{\begin{subequations}#1\end{subequations}}
\newcommand{\lb}[1]{\left\{ \begin{array}{ll} #1 \end{array} \right.}

\newcommand{\E}{\mathbb{E}}
 \newcommand{\nn}{\nonumber}
\newcommand{\cX}{\mathcal{X}}
\newcommand{\cA}{\mathcal{A}}
\newcommand{\cB}{\mathcal{B}}
\newcommand{\cC}{\mathcal{C}}
\newcommand{\cT}{\mathcal{T}}
\newcommand{\cH}{\mathcal{H}}
\newcommand{\cN}{\mathcal{N}}
\newcommand{\tgamma}{\tilde{\gamma}}

\newcommand{\defeq}{\buildrel\triangle\over =}

\newcommand{\pushright}[1]{\ifmeasuring@ #1 \else\omit\hfill$\displaystyle#1$\fi\ignorespaces}
\newcommand{\pushleft}[1]{\ifmeasuring@ #1 \else\omit$\displaystyle#1$\hfill\fi\ignorespaces}
\newcommand{\hs}[1]{\hspace{#1}}
\newcommand{\lp}{\left.}
\newcommand{\rp}{\right.}
\begin{document}
	%
	\title{A systematic process for evaluating structured perfect Bayesian equilibria in dynamic games with asymmetric information }
	%
	%
	%
	\author{Deepanshu~Vasal, Abhinav~Sinha and Achilleas~Anastasopoulos
	\thanks{The authors are with the Department
	of Electrical Engineering and Computer Science, University of Michigan, Ann
	Arbor, MI, 48105 USA e-mail: {\{dvasal, absi, anastas\}@umich.edu}.}
	\thanks{This work is supported in part by NSF grants CIF-1111061 and ECCS-1608361.}
	\thanks{This paper has originally appeared on arXiv.org on August 26, 2015 as working paper 1508.06269 and revised on September 14, 2016 as 1609.04221.}
	}
	
	
	%
	\maketitle
	\vspace{-1cm}
\begin{abstract}
We consider both finite-horizon and infinite-horizon versions of a dynamic game with $N$ selfish players who observe their types privately and take actions that are publicly observed. Players' types evolve as conditionally independent Markov processes, conditioned on their current actions. Their actions and types jointly determine their instantaneous rewards. In dynamic games with asymmetric information a widely used concept of equilibrium is perfect Bayesian equilibrium (PBE), which consists of a strategy and belief pair that simultaneously satisfy sequential rationality and belief consistency. In general, there does not exist a universal algorithm that decouples the interdependence of strategies and beliefs over time in calculating PBE.
In this paper, for the finite-horizon game with independent types we develop a two-step backward-forward recursive algorithm that sequentially decomposes the problem (w.r.t. time) to obtain a subset of PBEs, which we refer to as \textit{structured Bayesian perfect equilibria} (SPBE). In such equilibria, a player's strategy depends on its history only through a common public belief and its current private type. The backward recursive part of this algorithm defines an equilibrium generating function. Each period in the backward recursion involves solving a fixed-point equation on the space of probability simplexes for every possible belief on types. Using this
function, equilibrium strategies and beliefs are generated through a forward recursion.
We then extend this methodology to the infinite-horizon model, where we propose a time-invariant single-shot fixed-point equation, which in conjunction with a forward recursive step, generates the SPBE.
Sufficient conditions for the existence of SPBE are provided. 
With our proposed method, we find equilibria that exhibit \emph{signaling} behavior. 
This is illustrated with the help of a concrete public goods example. 

\end{abstract}


\section{Introduction}
Several practical applications involve dynamic interaction of strategic decision-makers with private and public observations.
Such applications include repeated online advertisement auctions, wireless resource sharing,  and energy markets.
In repeated online advertisement auctions, advertisers place bids for locations on a website to sell a product. These bids are calculated based on the value of that product, which is privately observed by the advertiser and past actions of other advertisers, which are observed publicly. Each advertiser's goal is to maximize its reward, which for an auction depends on the actions taken by others. In wireless resource sharing, players are allocated channels that interfere with each other. Each player privately observes its channel gain and takes an action, which can be the choice of modulation or coding scheme and also the transmission power. The reward it receives depends on the rate the player gets, which is a function of each player's channel gain and other players' actions (through the signal-to-interference ratio).
Finally, in an energy market, different suppliers bid their estimated power outputs to an independent system operator (ISO) that formulates the market mechanism to determine the prices assessed to the different suppliers. Each supplier wants to maximize its return, which depends on its cost of production of energy, which is its private information, and the market-determined prices which depend on all the bids.

Dynamical systems with strategic players are modeled as dynamic stochastic games, introduced by Shapley in~\cite{Sh53}. Discrete-time dynamic games with Markovian structure have been studied extensively to model many practical applications, in engineering as well as economics literature~\cite{BaOl98, FiVr12}. In dynamic games with perfect and symmetric information, subgame perfect equilibrium (SPE) is an appropriate equilibrium concept and there exists a backward recursive algorithm to find all the SPEs of these games (refer to~\cite{OsRu94, FuTi91book, samuelson2006} for a more elaborate discussion). Maskin and Tirole in \cite{MaTi01} introduced the concept of Markov perfect equilibrium (MPE) for dynamic games with symmetric information, where equilibrium strategies are dependent on some payoff relevant Markovian state of the system, rather than on the entire history. This is a refinement of the SPE. Some prominent examples of the application of MPE include~\cite{ErPa95, BeVa96, AcRo01}. Ericson and Pakes in~\cite{ErPa95} model industry dynamics for firms' entry, exit and investment participation, through a dynamic game with symmetric information, compute its MPE, and prove ergodicity of the equilibrium process. Bergemann and V{\" a}lim{\"a}ki in~\cite{BeVa96} study a learning process in a dynamic oligopoly with strategic sellers and a single buyer, allowing for price competition among sellers. They study MPE of the game and its convergence behavior. Acemo\u{g}lu and Robinson in~\cite{AcRo01} develop a theory of political transitions in a country by modeling it as a repeated game between the elites and the poor, and study its MPE.

In dynamic games with asymmetric information, and more generally in multi-player, dynamic decision problems (cooperative or non-cooperative) with asymmetric information, there is a signaling phenomenon that can occur, where a player's action reveals part of its private information to other players, which in turn affects their future payoff (see~\cite{KrSo94} for a survey of signaling models).\footnote{There are however instances where even though actions reveal private information, at equilibrium the signaling effect is non-existent~\cite{DoPa07,AlKaSi09} and~\cite[sec III.A]{NaGuLaBa14}. Thus, MPE is an appropriate equilibrium concept for such games. In~\cite{DoPa07}, authors extend the model of~\cite{ErPa95} where firms' set-up costs and scrap values are random and their private information. However, these are assumed to be i.i.d. across time and thus the knowledge of this private information in any period does not affect the future reward. In\cite{AlKaSi09},\cite[sec III.A]{NaGuLaBa14}, authors discuss games with one-step delayed information pattern, where all players get access to players' private information with delay one. In this case as well, signaling does not occur.} In one of the first works demonstrating signaling, a two-stage dynamic game was considered by Spence~\cite{Sp73}, where a worker signals her abilities to a potential employer using the level of education as a signal. Since then, this phenomenon has been shown in many settings, e.g., warranty as a signal for better quality of a product, in~\cite{Gr81}, larger deductible or partial insurance as a signal for better health of a person, in~\cite{Wi77,RoSt76}, and in evolutionary game theory, extra large antlers by a deer to signal fitness to a potential mate, in~\cite{Za75}.

For dynamic games with asymmetric information, where players' observations belong to different information sets, in order to calculate expected future rewards players need to form a belief on the observations of other players (where players need not have consistent beliefs). As a result, SPE or MPE,~\footnote{SPE and MPE are used for games where beliefs in the game are strategy-independent and consistent among players. Equivalently, these beliefs are derived from basic parameters of the problems and are not part of the definition of the equilibrium concept.} are not appropriate equilibrium concepts for such settings. There are several notions of equilibrium for such games, such as perfect Bayesian equilibrium (PBE), sequential equilibrium, and trembling hand equilibrium~\cite{OsRu94,FuTi91book}. Each of these equilibrium notions consist of an assessment, i.e., a strategy and a belief profile for the entire time horizon. The equilibrium strategies are optimal given the equilibrium beliefs and the equilibrium beliefs are derived from the equilibrium strategy profile using Bayes' rule (whenever possible), with some equilibrium concepts requiring further refinements. Thus there is a cyclical requirement of beliefs being consistent with strategies, which are in turn optimal given the beliefs, and finding such equilibria can be thought of as being equivalent to solving a fixed point equation in the space of strategy and belief profiles over the entire time horizon. Furthermore, these strategies and beliefs are functions of histories and thus their domain grows exponentially in time, which makes the problem computationally intractable. To date, there is no universal algorithm that provides simplification by decomposing the aforementioned fixed-point equation for calculating PBEs.

Some practically motivated work in this category is the work in~\cite{Ba92,BiHiWe92,SmSo02,DePeSi15}. Authors in \cite{Ba92,BiHiWe92,SmSo02} study the problem of social learning with sequentially-acting selfish players who act exactly once in the game and make a decision to adopt or reject a trend based on their estimate of the system state. Players observe a private signal about the system state and publicly observe actions of past players. The authors analyze PBE of the dynamic game and study the convergent behavior of the system under an equilibrium, where they show occurrence of herding. Devanur et al. in~\cite{DePeSi15} study PBE of a repeated sales game where a single buyer has a valuation of a good, which is its private information, and a seller offers to sell a fresh copy of that good in every period through a posted price.
\vspace{-0.1cm}

\subsection{Contributions}
In this paper, we present a sequential decomposition methodology for calculating a subset of all PBEs for finite and infinite horizon dynamic games with asymmetric information.
Our model, consists of strategic players having types that evolve as conditionally independent Markov controlled processes. Players observe their types privately and actions taken by all  players are observed publicly. Instantaneous reward for each player depends on everyone's types and actions. The proposed methodology provides a decomposition of the interdependence between beliefs and strategies in PBE and enables a systematic evaluation of a subset of PBE, namely \textit{structured perfect Bayesian equilibria} (SPBE).
Furthermore, we show that all SPBE can be computed using this methodology.
Here SPBE are defined as equilibria with players strategies based on their current private type and a set of beliefs on each player's current private type, which is common to all the players and whose domain is time-invariant. The beliefs on players' types are such that they can be updated individually for each player and sequentially w.r.t. time. The model allows for signaling amongst players as beliefs depend on strategies.

Our motivation for considering SPBE stems from ideas in decentralized team problems and specifically the works of Ho~\cite{Ho80} and Nayyar et al.~\cite{NaMaTe11}. We utilize the agent-by-agent approach in~\cite{Ho80} to motivate a Markovian structure where players' strategies depend only on their current types. In addition, we utilize the common information based approach introduced in~\cite{NaMaTe11} to summarize the common information into a common belief on players' private types. Even though these ideas motivate the special structure of our equilibrium strategies, they can not be applied in games to evaluate SPBE because they have been developed for dynamic teams and are incompatible with equilibrium notions. Our main contribution is a new construction based on which SPBE can be systematically evaluated.

Specifically, for the finite horizon model, we provide a two-step algorithm involving a backward recursion followed by a forward recursion. The algorithm works as follows. In the backward recursion, for every time period, the algorithm finds an equilibrium generating function defined for all possible common beliefs at that time. This involves solving an one-step fixed point equation on the space of probability simplexes. Then, the equilibrium strategies and beliefs are obtained through a forward recursion using the equilibrium generating function obtained in the backward step and the Bayes update rule. The SPBE that are developed in this paper are analogous to MPEs (for games with symmetric information) in the sense that players choose their actions based on beliefs that depend on common information, and private types, both of which have Markovian dynamics.

For the infinite horizon model, instead of the backwards recursion step, the algorithm solves a single-shot time invariant fixed-point equation involving both an equilibrium generating function and an equilibrium reward-to-go function.
We show that using our method, existence of SPBE in the asymmetric information dynamic game is guaranteed if the aforementioned fixed-point equation admits a solution. We provide sufficient conditions under which this is true.  	
We demonstrate our methodology of finding SPBE through a multi-stage public goods game, whereby we observe the aforementioned signaling effect at equilibrium.
\vspace{-0.2cm}
\subsection{Relevant Literature}
Related literature on this topic include \cite{NaGuLaBa14, GuNaLaBa14} and \cite{OuTaTe15}. Nayyar et al. in \cite{NaGuLaBa14, GuNaLaBa14} consider dynamic games with asymmetric information. There is an underlying controlled Markov process and players jointly observe part of the process and whilst making additional private observations. It is shown that the considered game with asymmetric information, under certain assumptions, can be transformed to another game with symmetric information. A backward recursive algorithm is provided to find MPE of the transformed game. For this strong equivalence to hold, authors in \cite{NaGuLaBa14, GuNaLaBa14} make a critical assumption in their model: based on the common information, a player's posterior beliefs about the system state and about other players' information are independent of the past strategies used by the players. This leads to all strategies being non-signaling. Our model is different from this since we assume that the underlying state of the system has independent components, each constituting a player's private type. However, we do not make any assumption regarding update of beliefs and allow the belief state to depend on players' past strategies, which in turn allows the possibility of signaling in the game.

Ouyang et al. in \cite{OuTaTe15} consider a dynamic oligopoly game with strategic sellers and buyers. Each seller privately observes the valuation of their good, which is assumed to have independent Markovian dynamics, thus resulting in a dynamic game of asymmetric information. The common belief is strategy dependent and the authors consider equilibria based on this common information belief.
%
%
It is shown that if all other players play actions based on the common belief and their private information using equilibrium strategies, and if all players use equilibrium belief update function, then player $i$ faces a Markov decision process (MDP) with respect to its action with state being the common belief and its private type.
Thus calculating equilibrium boils down to solving a fixed-point equation on belief update functions and strategies of all players. Existence of such equilibrium is shown for a degenerate case where players act myopically at equilibrium and the equilibrium itself is non-signaling.

Other than the common information based approach, Li et al.~\cite{shamma14} consider a finite horizon zero-sum dynamic game, where at each time only one player out of the two knows the state of the system. The value of the game is calculated by formulating an appropriate linear program. Cole et al.~\cite{cole2001} consider an infinite horizon discounted reward dynamic game where actions are only privately observable. They provide a fixed-point equation for calculating a subset of sequential equilibrium, which is referred to as Markov private equilibrium (MPrE). In MPrE strategies depend on history only through the latest private observation.
\vspace{-0.2cm}
\subsection{Notation}
We use uppercase letters for random variables and lowercase for their realizations. For any variable, subscripts represent time indices and superscripts represent player indices. We use notation $ -i$ to represent all players other than player $i$ i.e. $ -i = \{1,2, \ldots i-1, i+1, \ldots, N \}$. We use notation $A_{t:t'}$ to represent the vector $(A_t, A_{t+1}, \ldots A_{t'})$ when $t'\geq t$ or an empty vector if $t'< t$. We use $A_t^{-i}$ to mean $(A^1_t, A^2_{t}, \ldots, A_t^{i-1}, A_t^{i+1} \ldots, A^N_{t})$ . We remove superscripts or subscripts if we want to represent the whole vector, for example $ A_t$  represents $(A_t^1, \ldots, A_t^N) $. In a similar vein, for any collection of sets $(\cX^i)_{i \in \cN}$, we denote $\times_{i\in\cN} \cX^i$ by $\cX$. We denote the indicator function of a set $A$ by $I_{A}(\cdot)$.
For any finite set $\mathcal{S}$, $\Delta(\mathcal{S})$ represents the space of probability measures on $\mathcal{S}$ and $|\mathcal{S}|$ represents its cardinality. We denote by $\mP^g$ (or $\mE^g$) the probability measure generated by (or expectation with respect to) strategy profile $g$. We denote the set of real numbers by $\mathbb{R}$. For a probabilistic strategy profile of players $(\beta_t^i)_{i\in \cN}$ where the probability of action $a_t^i$ conditioned on $a_{1:t-1},x_{1:t}^i$ is given by $\beta_t^i(a_t^i|a_{1:t-1},x_{1:t}^i)$, we use the notation $\beta_t^{-i}(a_t^{-i}|a_{1:t-1},x_{1:t}^{-i})$ to represent $\prod_{j\neq i} \beta_t^j(a_t^j|a_{1:t-1},x_{1:t}^j)$.
All equalities/inequalities involving random variables are to be interpreted in the \emph{a.s.} sense.
For mappings with range function sets $f: \cA \rightarrow (\cB \rightarrow \cC)$ we use square brackets $f[a] \in \cB \rightarrow \cC$ to denote the image of $a\in\cA$ through $f$ and parentheses $f[a](b) \in \cC$ to denote the image of $b\in\cB$ through $f[a]$.
A controlled Markov process with state $X_t$, action $A_t$, and horizon $\cT$ is denoted by $(X_t,A_t)_{t\in\cT}$.

The paper is organized as follows. In Section~\ref{sec:Model}, we present the model for games with finite and infinite horizon.
Section~\ref{sec:StructuralResults} serves as motivation for focusing on SPBE.
In Section~\ref{sec:fh}, for finite-horizon games, we present a two-step backward-forward recursive algorithm to construct a strategy profile and a sequence of beliefs, and show that it is a PBE of the dynamic game considered.  In Section~\ref{secih}, we extend that methodology to infinite-horizon games. Section~\ref{secexample} discusses concrete example of a public goods game with two players and results are presented for both, finite and infinite horizon versions of the example. All proofs are provided in appendices.
\section{Model and Preliminaries}
\label{sec:Model}
We consider a discrete-time dynamical system with $N$ strategic players in the set $ \cN \defeq \{1,2, \ldots N \}$. {We consider two cases: finite horizon $\mathcal{T} \defeq \{1, 2, \ldots T\}$ with perfect recall and infinite horizon with perfect recall}. The system state is $X_t \defeq (X_t^1, X_t^2, \ldots X_t^N)$, where $X_t^i \in \cX^i$ is the type of player $i$ at time $t$, which is perfectly observed and is its private information. Players' types evolve as conditionally independent, controlled Markov processes such that
\seq{
\eq{
\mP(x_1) &= \prod_{i=1}^N Q^i_1(x_1^i)\\
\mP(x_t|x_{1:t-1}, a_{1:t-1}) &= \mP(x_t|x_{t-1} , a_{t-1})\\
&= \prod_{i=1}^N Q_t^{i}(x_t^i|x_{t-1}^i, a_{t-1}),
}
}
where $Q^i_t$ are known kernels. Player $i$ at time $t$ takes action $a_t^i \in \cA^i$ on observing the actions {$a_{1:t-1} = (a_k)_{k=1,\ldots,t-1}$ where $ a_k = \left( a_k^j \right)_{j \in \mN} $}, which is common information among players, and the types $x_{1:t}^i$, which it observes privately. The sets $\cA^i, \cX^i $ are assumed to be finite. Let $g^i = ( g^i_t)_{t \in \mathcal{T}}$ be a probabilistic strategy of player $i$  where $g^i_t : \cA^{t-1}\times (\cX^{i})^t \to \Delta(\cA^i)$ such that player $i$ plays action $A_t^i$ according to $ A_t^i \sim g^i_t(\cdot|a_{1:t-1},x_{1:t}^i)$. Let $ g \defeq(g^i)_{i\in \cN}$ be a strategy profile of all players. At the end of interval $t$, player $i$ receives an instantaneous reward $R_t^i(x_t,a_t)$. {To preserve the information structure of the problem, we assume that players do not observe their rewards until the end of game.\footnote{Alternatively, we could have assumed instantaneous reward of a player to depend only on its own type, i.e. be of the form $R_t^i(x_t^i,a_t)$, and have allowed rewards to be observed by the players during the game as this would not alter the information structure of the game} The reward functions and state transition kernels are common knowledge among the players.} For the finite-horizon problem, the objective of player $i$ is to maximize its total expected reward
\eq{ J^{i,g} \defeq \E^g \left\{ \sum_{t=1}^T R_t^i(X_t,A_t) \right\}.
}
For the infinite-horizon case, the transition kernels $ Q_t^i $ are considered to not depend on time $ t $. We also substitute $ R_t^i(X_t,A_t) = \delta^t R^i(X_t,A_t) $ take $ \lim_{T \rightarrow \infty} $ in the above equation, where $ \delta \in [0,1) $ is the common discount factor and $ R^i $ is the time invariant stage reward function for player $ i $.
With all players being strategic, this problem is modeled as a dynamic game, $\mathfrak{D}_T$ for finite horizon and $ \mathfrak{D}_\infty $ for infinite horizon, with asymmetric information and simultaneous moves.

\subsection{Preliminaries}
\label{sec:Result_A}
Any history of this game at which players take action is of the form $h_t = (a_{1:t-1},x_{1:t})$. Let $\mathcal{H}_t$ be the set of such histories, $\mathcal{H}^T \defeq \cup_{t=0}^T \mathcal{H}_t $ be the set of all possible such histories in finite horizon and $\mathcal{H}^\infty \defeq \cup_{t=0}^\infty \mathcal{H}_t $ for infinite horizon. At any time $t$ player $i$ observes $h^i_t = (a_{1:t-1},x_{1:t}^{i})$ and all players together have $h^c_t = a_{1:t-1}$ as common history. Let $\mathcal{H}^i_t$ be the set of observed histories of player $i$ at time $t$ and $\mathcal{H}^c_t$ be the set of common histories at time $t$. An appropriate concept of equilibrium for such games is PBE \cite{FuTi91book}, which consists of a pair $(\beta^*,\mu^*)$ of strategy profile $\beta^* = (\beta_t^{*,i})_{t \in \mathcal{T},i\in \cN}$ where $\beta_t^{*,i} : \mathcal{H}_t^i \to \Delta(\cA^i)$ and a belief profile $\mu^* = (^i\mu_t^{*})_{t \in \mathcal{T},i\in \cN}$ where $^i\mu_t^{*}: \mathcal{H}^i_t \to \Delta(\mathcal{H}_t)$ that satisfy sequential rationality so that $\forall i \in \cN,  t \in \mathcal{T},  h^{i}_t \in \mathcal{H}^i_t, {\beta^{i}}$
\begin{equation}
W_t^{i,\beta^{*,i},T}(h_t^i) \ge W_t^{i,\beta^i,T}(h_t^i)
\end{equation}
where the reward-to-go  is defined as
\begin{equation}
\hs{-0.2cm}W_t^{i,\beta^i,T}(h_t^i) \triangleq \E^{{\beta}^{i} \beta^{*,-i},\, ^i\mu_t^*[h_t^i]}\left\{ \sum_{n=t}^T R_n^i(X_n, A_n)\big\lvert  h^i_t\right\}, \;\;   \label{eq:seqeq}
\end{equation}
and the beliefs satisfy some consistency conditions as described in~\cite[p. 331]{FuTi91book}. Similarly, for the game $ \mathfrak{D}_\infty $ PBE $(\beta^\star,\mu^\star)$ requires: $\forall i \in \cN,  t \ge 1,  h^{i}_t \in \mathcal{H}^i_t, {\beta^{i}}$
\begin{equation}
W_t^{i,\beta^{*,i}}(h_t^i) \ge W_t^{i,\beta^i}(h_t^i)
\end{equation}
where the reward-to-go is
\begin{equation}
	W_t^{i,\beta^i}(h_t^i) \triangleq \E^{{\beta}^{i} \beta^{*,-i},\, ^i\mu_t^*[h_t^i]}\left\{ \sum_{n=t}^\infty R_n^i(X_n, A_n)\big\lvert  h^i_t\right\}.   \label{eq:seqeqinf}
\end{equation}
In general, a belief for player $i$ at time $t$, $^i\mu_t^{*}$ is defined on history $h_t = (a_{1:t-1},x_{1:t}) $ given its private history $h^i_t = (a_{1:t-1},x_{1:t}^{i})$. Here player $i$'s private history $h^i_t=(a_{1:t-1},x_{1:t}^i)$ consists of a public part $h_t^c=a_{1:t-1}$ and a private part $x_{1:t}^i$.
At any time $t$, the relevant uncertainty player $i$ has is about other players' types $x_{1:t}^{-i} \in \times_{n=1}^t \left( \times_{j \ne i} \mX^j \right)$  and their future actions.
In our setting, due to independence of types, and given the common history $h_t^c$, player $i$'s type history $x_{1:t}^i$ does not provide any additional information about $x_{1:t}^{-i}$, as will be shown later. For this reason we consider beliefs that are functions of each player's history $h^i_t$ only through the common history $h^c_t$.
Hence, for each player $i$, its belief for each history $h^c_t=a_{1:t-1}$ is derived from a common belief $\mu^*_t[a_{1:t-1}]$. Furthermore, as will be shown later, this belief factorizes into a product of marginals $\prod_{j\in\cN} \mu^{*,j}_t[a_{1:t-1}] $.
Thus we can sufficiently use the system of beliefs, $\mu^*=(\underline{\mu}^*_t)_{t\in\mathcal{T}}$, where  $\underline{\mu}^*_t = (\mu^{*,i}_t)_{i \in \mN} $, and $\mu^{*,i}_t: \mathcal{H}^c_t \to \Delta(\cX^i)$, with the understanding that player $i$'s belief on $x_t^{-i}$ is  $\mu^{*,-i}_t[a_{1:t-1}](x_t^{-i})=\prod_{j\neq i} \mu^{*,j}_t[a_{1:t-1}](x_t^j)$.
Under the above structure, all consistency conditions that are required for PBEs~\cite[p. 331]{FuTi91book} are automatically satisfied.


\section{Motivation for structured equilibria}
\label{sec:StructuralResults}
In this section, we present structural results for the considered dynamical process that serve as a motivation for finding SPBE of the underlying game $\mathfrak{D}_T$.
Specifically, we define a belief state based on common information history and show that any reward profile that can be obtained through a general strategy profile can also be obtained through strategies that depend on this belief state and players' current types, which are their private information. These structural results are inspired by the analysis of decentralized team problems, which serve as guiding principles to design our equilibrium strategies. While these structural results provide intuition and the required notation, they are not directly used in the proofs for finding SPBE later in Section~\ref{sec:Result}.

At any time $t$, player $i$ has information $(a_{1:t-1}, x_{1:t}^{i})$ where $a_{1:t-1}$ is the common information among players, and $x_{1:t}^{i}$ is the private information of player $i$. Since $(a_{1:t-1}, x_{1:t}^{i})$ increases with time, any strategy of the form $A_t^i \sim g^i_t(\cdot|a_{1:t-1},x_{1:t}^{i})$ becomes unwieldy. Thus it is desirable to have an information state in a time-invariant space that succinctly summarizes $(a_{1:t-1}, x_{1:t}^{i})$, and that can be sequentially updated. We first show in Lemma~\ref{fact:G2S} that given the common information $a_{1:t-1}$ and its current type $x_t^i$, player $i$ can discard its type history $x_{1:t-1}^i$ and play a strategy of the form $A_t^i \sim s^i_t(\cdot|a_{1:t-1},x_{t}^{i})$. Then in Lemma~\ref{fact:L1}, we show that $a_{1:t-1}$ can be summarized through a belief $\pi_t$, defined as follows. For any strategy profile $g$, belief $\pi_t $ on $X_t$,  $\pi_t \in \Delta (\cX)$,  is defined as $\pi_t (x_t) \defeq \mP^g(X_t=x_t|a_{1:t-1}), \; \forall x_t\in \cX$. We also define the marginals $\pi_t^{{i}} (x_t^i) \defeq \mP^g(X_t^i=x_t^i|a_{1:t-1}), \; \forall x_t^i \in \cX^i$.

For player $i$, we use the notation $g$ to denote a general policy of the form $A_t^i\sim g_t^i(\cdot|a_{1:t-1}, x_{1:t}^{i})$, notation $s$, where $s_t^i: \cA^{t-1} \times \cX^i \to \Delta(\cA^i)$, to denote a policy of the form $A_t^i \sim s_t^i(\cdot|a_{1:t-1},x_t^i)$, and notation $m$, where $m_t^i: \Delta(\times_{i\in\cN} \cX^i) \times \cX^i \to \Delta(\cA^i)$, to denote a policy of the form $A_t^i \sim m_t^i(\cdot|\pi_t,x_t^i)$. It should be noted that since $\pi_t$ is a function of random variables $a_{1:t-1}$, $m$ policy is a special type of $s$ policy, which in turn is a special type of $g$ policy.

Using the agent-by-agent approach~\cite{Ho80}, we show in Lemma~\ref{fact:G2S} that any expected reward profile of the players that can be achieved by any general strategy profile $g$ can also be achieved by a strategy profile $s$.
\begin{lemma}
Given a fixed strategy $g^{-i}$ of all players other than player $i$ and for any strategy $g^i$ of player $i$, there exists a strategy $s^i$ of player $i$ such that $\forall t \in \mathcal{T}, x_t\in \cX, a_t\in \cA,$
\eq{\mP^{s^i g^{-i}}(x_t, a_t) &= \mP^{g^ig^{-i}}(x_t, a_t) \;\;\;\;\;
}
which implies $ J^{i,s^ig^{-i}} = J^{i,g^ig^{-i}}$.\label{fact:G2S}
\end{lemma}
\begin{IEEEproof}
	Please see Appendix~\ref{app:A1}.
\end{IEEEproof}
Since any $s^i$ policy is also a $g^i$ type policy, the above lemma can be iterated over all players which implies that for any $g$ policy profile there exists an $s$ policy profile that achieves the same reward profile i.e., $(J^{i,s})_{i\in \cN} = (J^{i,g})_{i\in \cN}$.

Policies of types $s$ still have increasing domain due to increasing common information $a_{1:t-1}$. In order to summarize this information, we take an equivalent view of the system dynamics through a common agent, as taken in \cite{nayyar2013}. The common agent approach is a general approach that has been used extensively in dynamic team problems \cite{Ma13, MaTe08, Nate08, VaAn14}. Using this approach, the problem can be equivalently described as follows: player $i$ at time $t$ observes $a_{1:t-1}$ and takes action $\gamma_t^i$, where $\gamma_t^i :  \cX^i \to \Delta(\cA^i)$ is a partial (stochastic) function from its private information $x_t^i$ to  $a_t^i$, of the form $A_t^i \sim \gamma_t^i(\cdot|x_t^i)$. These actions are generated through some policy $\psi^i = (\psi^i_t)_{t \in \mathcal{T}}$, $\psi^i_t : \cA^{t-1} \to \left\{  \cX^i \to \Delta(\cA^i) \right\}$, that operates on the common information $a_{1:t-1}$ such that $\gamma_t^i = \psi_t^i[a_{1:t-1}]$. Then any policy of the form $A_t^i \sim s_t^i(\cdot|a_{1:t-1},x_t^{i})$ is equivalent to $A_t^i \sim \psi^i_t[a_{1:t-1}] (\cdot|x_t^i)$.

We call a player $i$'s policy through common agent to be of type $\psi^i$ if its actions $\gamma^i_t$ are taken as $\gamma^i_t = \psi^i_t[a_{1:t-1}]$. We call a player $i$'s policy through common agent to be of type $\theta^i$ where $\theta^i_t : \Delta(\cX) \to \left\{ \cX^i \to \Delta(\cA^i) \right\}$, if its actions $\gamma^i_t$ are taken as $\gamma^i_t =\theta^i_t[\pi_t]$. A policy of type $\theta^i$ is also a policy of type $\psi^i$. There is a one-to-one correspondence between policies of type $s^i$ and of type $\psi^i$ and  between policies of type $m^i$ and of type $\theta^i$. In summary, the notation for the various functional form of strategies is
\begin{subequations}
	\begin{align}
	&A_t^i \sim s_t^i(\cdot|a_{1:t-1},x_t^i) \qquad 	
	A_t^i \sim \psi^i_t[a_{1:t-1}] (\cdot|x_t^i),
	\\
	&A_t^i \sim m_t^i(\cdot|\pi_t,x_t^i) \qquad \quad
	A_t^i \sim \theta^i_t[\pi_t] (\cdot|x_t^i). \label{eq:m_policies}
	\end{align}
\end{subequations}

In the following lemma, we show that the space of profiles of type $s$ is outcome-equivalent to the space of profiles of type $m$.
\begin{lemma}
\label{fact:L1}
For any given strategy profile $s$ of all players, there exists a strategy profile $m$ such that
\eq{\mP^m(x_t, a_t) &= \mP^s(x_t, a_t) \;\;\;\;\forall t \in \mathcal{T}, x_t\in \cX, a_t\in \cA ,
}
which implies $ (J^{i,m})_{i\in \cN} = (J^{i,s})_{i\in \cN} $.
\label{fact:S2M}
\end{lemma}
\begin{IEEEproof}
Please see Appendix~\ref{app:A2}. 
\end{IEEEproof}

The above two lemmas show that any reward profile that can be generated through a policy profile of type $g$ can also be generated through a policy profile of type $m$.
This is precisely the motivation for using SPBE which are equilibria based on policies of type $m$.
It should be noted that the construction of $s^i$ depends only on $g^i$ (as shown in~\eqref{eq:defSi_main}), %
while the construction of $m^i$ depends on the whole policy profile $g$ and not just on $g^i$, since the construction of $\theta^i$ depends on $\psi$ in~\eqref{eq:defMi}. %
Thus any unilateral deviation of player $i$ in $g$ policy profile does not necessarily translate to unilateral deviation of player $i$ in the corresponding $m$ policy profile. Therefore $g$ being an equilibrium of the game (in some appropriate notion) does not necessitate the corresponding $m$ also being an equilibrium. Thus the set of equilibria of type $ g $ contains those of type $ m $ but not vice-versa (in general); characterizing the relationship between the two sets of equilibria is an interesting open problem.

We end this section by noting that although finding general PBEs of type $g$ of the games $\mathfrak{D}_T$ or $ \mathfrak{D}_\infty $ would be a desirable goal, since the space of strategies is growing exponentially with time, it would be computationally intractable.
However, Lemmas~1 and 2 suggest that strategies of type $m$ form a rich class that achieves every possible reward profile. Since these strategies are functions of beliefs $\pi_t$ that lie in a time-invariant space and are easily updatable, equilibria of this type are potential candidates for computation through backward recursion. Our goal is to devise an algorithm to find structured equilibria of type $m$ of the dynamic games $\mathfrak{D}_T$ or $ \mathfrak{D}_\infty $.

\begin{definition}[SPBE]
	A structured perfect Bayesian equilibrium is a PBE of the considered dynamic game where at any time $ t $, for any agent $ i $, its equilibrium strategy $ \beta_t^{\ast,i} $ is of type $ m $ (as in~\eqref{eq:m_policies}).
\end{definition}
\vspace{-0.2cm} 

\section{A Methodology for SPBE computation in finite horizon}
\label{sec:Result} \label{sec:fh}
In this section we consider the finite horizon dynamic game $ \mathfrak{D}_{T} $. In the previous section,  (specifically in Claim~\ref{claim:CondInd}, included in the proof of Lemma~\ref{fact:L1} in Appendix~\ref{app:A2}), it is shown that due to the independence of types and their evolution as independent controlled Markov processes, for any strategy of the players, the joint common belief can be factorized as a product of its marginals i.e., $\pi_t(x_t) = \prod_{i=1}^N \pi_t^{i}(x_t^i), \forall x_t$. Since in this paper, we only deal with such joint beliefs, to accentuate this independence structure, we define $\underline{\pi}_t \in \times_{i\in \cN} \Delta(\cX^i)$ as vector of marginal beliefs where $\underline{\pi}_t := (\pi^i_t)_{i\in \cN}$. In the rest of the paper, we will use $\underline{\pi}_t$ instead of $\pi_t$ whenever appropriate, where of course, $\pi_t$ can be constructed from $\underline{\pi}_t$. Similarly, we define the vector of belief updates as $\underline{F}(\underline{\pi},\gamma,a) := (F^i(\pi^i,\gamma^i,a))_{i \in \cN}$ where (using Bayes rule)
\eq{\label{eq:F_update}
F^i(\pi^i,\gamma^i,a)(x_{t+1}^i)=  \left\{
\begin{array}{ll}
\frac{  \sum_{x^i_t} \pi^{{i}}(x_t^i) \gamma^i(a^i|x_t^i)Q_t^i(x_{t+1}^i|x_t^i, a) }{ \sum_{\tilde{x}_t^i}\pi^{{i}}(\tilde{x}_t^i)  \gamma^i(a^i|\tilde{x}_t^i)}   &\mbox{if } \sum_{\tilde{x}_t^i}\pi^{{i}}(\tilde{x}_t^i)  \gamma^i(a^i|\tilde{x}_t^i) > 0 \\
\sum_{x_t^i}\pi^i(x_t^i)Q_t^i(x_{t+1}^i|x_t^i,a)  &\mbox{if } \sum_{\tilde{x}_t^i}\pi^{{i}}(\tilde{x}_t^i)  \gamma^i(a^i|\tilde{x}_t^i) = 0.
\end{array}
\right.
}
The update function $ F^i $ defined above depends on time $ t $ through the kernel $Q_t^i$ (for the finite horizon model). For notational simplicity we suppress this dependence on $ t $.
%
%
We also change the notation of policies of type $m$ and $\theta$ as follows, so they depend on $\upi_t$ instead of $\pi_t$
\begin{subequations}
\eq{
&m_t^i: \times_{i\in\cN} \Delta(\cX^i)\times\cX^i \to \Delta(\cA^i)
\\
&\theta^i_t : \times_{i\in\cN} \Delta(\cX^i) \to \left\{ \cX^i \to \Delta(\cA^i) \right\}.
}
\end{subequations}

In the following we present a backward-forward algorithm that evaluates SPBE. As will be shown in Theorem~\ref{thm:2}, this is a ``canonical'' methodology, in the sense that all SPBE can be generated this way.

\subsection{Backward Recursion} \label{sec:fhbr}

In this section, we define an equilibrium generating function $\theta=(\theta^i_t)_{i\in\cN,t\in\mathcal{T}}$, where $\theta^i_t : \times_{i\in\cN} \Delta(\cX^i) \to \left\{\cX^i \to \Delta(\cA^i) \right\}$. In addition, we define a sequence of reward-to-go functions of player $i$ at time $t$,  $(V_t^i)_{i\in \cN, t\in \{ 1,2, \ldots T+1\}}$, where $V_t^i : \times_{i\in\cN} \Delta(\cX^i) \times \cX^i \to \mathbb{R}$.
These quantities are generated through a backward recursive way, as follows.
\begin{itemize}
\item[1.] Initialize $\forall \underline{\pi}_{T+1}\in \times_{i\in\cN} \Delta(\cX^i), x_{T+1}^i\in \cX^i$,
\eq{
V^i_{T+1}(\underline{\pi}_{T+1},x_{T+1}^i) \defeq 0.   \label{eq:VT+1}
}

\item[2.] For $t = T,T-1, \ldots 1, \ \forall \underline{\pi}_t \in \times_{i\in\cN} \Delta(\cX^i), \pi_t = \prod_{i\in\cN}\pi_t^i $, let $\theta_t[\underline{\pi}_t] $ be generated as follows. Set $\tilde{\gamma}_t = \theta_t[\underline{\pi}_t]$, where $\tilde{\gamma}_t$ is the solution, if it exists,\footnote{The problem of existence in this step will be discussed in Section~\ref{sec:existence}.} of the following fixed-point equation, $\forall i \in \cN,x_t^i\in \cX^i$,
  \eq{
 \tilde{\gamma}^{i}_t(\cdot|x_t^i) \in \arg\max_{\gamma^i_t(\cdot|x_t^i)} &\E^{\gamma^i_t(\cdot|x_t^i) \tilde{\gamma}^{-i}_t,\,\pi_t} \left\{ R_t^i(X_t,A_t) + V_{t+1}^i (\uF(\underline{\pi}_t, \tilde{\gamma}_t, A_t), X_{t+1}^i) \big\lvert x_t^i \right\} , \label{eq:m_FP}
  }
 where expectation in \eqref{eq:m_FP} is with respect to random variables $(X_t^{-i},A_t, X_{t+1}^i)$ through the measure
$\pi_t^{-i}(x_t^{-i})\gamma^i_t(a^i_t|x_t^i) \tilde{\gamma}^{-i}_t(a^{-i}_t|x_t^{-i})Q_{t+1}^i(x_{t+1}^i|x_t^i,a_t)$ and $\uF$ is defined above.

 Furthermore, using the quantity $\tilde{\gamma}_t$ found above, define
\eq{
V^i_{t}(\underline{\pi}_t,x_t^i) \defeq & \E^{\tilde{\gamma}^{i}_t(\cdot|x_t^i) \tilde{\gamma}^{-i}_t,\, \pi_t}\left\{ {R}_t^i (X_t,A_t) + V_{t+1}^i (\uF(\underline{\pi}_t, \tilde{\gamma}_t, A_t), X_{t+1}^i)\big\lvert  x_t^i \right\}.  \label{eq:Vdef}
}
   \end{itemize}

It should be noted that in \eqref{eq:m_FP}, $\tilde{\gamma}_t^i$ is not the outcome of a maximization operation as is the case in a best response equation of a Bayesian Nash equilibrium. Rather \eqref{eq:m_FP} is a different fixed point equation. This is because the maximizer $\tilde{\gamma}^i_t$ appears in both, the left-hand-side and the right-hand-side of the equation (in the belief update $\uF(\underline{\pi}_t, \tilde{\gamma}_t, A_t)=(F^i(\pi^i_t,\tilde{\gamma}^i_t,A_t))_{i\in\cN}$).
This distinct construction
is pivotal in the proof of Theorem~\ref{Thm:Main}, as will be further elaborated in the Discussion section.


\vspace{-0.35cm}

\subsection{Forward Recursion}\label{sec:fr}


As discussed above, a pair of strategy and belief profile $(\beta^{*}, \mu^{*})$ is a PBE if it satisfies \eqref{eq:seqeq}. Based on $\theta$ defined above in \eqref{eq:VT+1}--\eqref{eq:Vdef}, we now construct a set of strategies $\beta^*$ and beliefs $\mu^*$ for the game $\mathfrak{D}_T$ in a forward recursive way, as follows.\footnote{As discussed in the preliminaries subsection on Section~\ref{sec:Model}, the equilibrium beliefs in SPBE, $\mu_t^*$ are functions of each player's history $h^i_t$ only through the common history $h^c_t$ and are the same for all players.} As before, we will use the notation $\underline{\mu}_t^*[a_{1:t-1}] := (\mu_t^{*,i}[a_{1:t-1}])_{i\in \cN}$, where $\mu_t^{*,{i}}[a_{1:t-1}]$ is a belief on $x_t^i$, and  $\mu_t^*[a_{1:t-1}]$ can be constructed from $\underline{\mu}_t^*[a_{1:t-1}]$ as $\mu_t^*[a_{1:t-1}](x_t) = \prod_{i=1}^N\mu_t^{*,i}[a_{1:t-1}](x_t^i),\; \forall a_{1:t-1}\in \mathcal{H}_t^c$.
\begin{itemize}
\item[1.] Initialize at time $t=1$,
\eq{
\mu^{*}_1[\phi](x_1) &:= \prod_{i=1}^N Q_1^i(x_1^i). \label{eq:mu*def0}
}
\item[2.] For $t =1,2 \ldots T, \forall i \in \cN, a_{1:t}\in \cH_{t+1}^c, x_{1:t}^i \in(\cX^i)^t$
\eq{
\beta_{t}^{*,i}(a_{t}^i|a_{1:t-1},x_{1:t}^i)&= \beta_{t}^{*,i}(a_{t}^i|a_{1:t-1},x_{t}^i)  \nn\\
&:= \theta_{t}^i[\underline{\mu}_{t}^*[a_{1:t-1}]](a^i_{t}|x_{t}^i) \label{eq:beta*def}
}
and
\eq{
\mu^{*,i}_{t+1}[a_{1:t}] &:= F^i(\mu_t^{*,i}[a_{1:t-1}], \theta_t^i[\underline{\mu}_t^*[a_{1:t-1}]], a_t) \label{eq:mu*def}
}

\end{itemize}
where $F^i$ is defined in~\eqref{eq:F_update}.

We now state our main result.

\begin{theorem}
\label{Thm:Main}
A strategy and belief profile $(\beta^*,\mu^*)$, constructed through the backward-forward recursion algorithm is a PBE of the game, i.e.,
$\forall i \in \cN,t \in \mathcal{T}, a_{1:t-1} \in \mathcal{H}_t^c, x_{1:t}^i \in (\cX^i)^t, \beta^i$,
\eq{
\E^{\beta_{t:T}^{*,i} \beta_{t:T}^{*,-i},\,\mu_{t}^{*}[a_{1:t-1}]} \left\{ \sum_{n=t}^T R_n^i(X_n,A_n) \big\lvert  a_{1:t-1}, x_{1:t}^i \right\} 
\geq
\E^{\beta_{t:T}^{i} \beta_{t:T}^{*,-i},\, \mu_{t}^{*}[a_{1:t-1}]} \left\{ \sum_{n=t}^T R_n^i(X_n,A_n) \big\lvert  a_{1:t-1}, x_{1:t}^i \right\}. \label{eq:prop}
}
\end{theorem}
\begin{IEEEproof}
Please see Appendix~\ref{app:B}.
\end{IEEEproof}
We emphasize that even though the backward-forward algorithm presented above finds a class of equilibrium strategies that are structured, the unilateral deviations of players in \eqref{eq:prop} are considered in the space of general strategies, i.e., the algorithm does not make any bounded rationality assumptions.

The following result shows that the backward-forward construction described above is ``canonical'', in the sense that all SPBE can be found through this methodology. Clearly, an SPBE can be defined as a PBE $(\beta^*,\mu^*)$ of the game that is generated through forward recursion in \eqref{eq:mu*def0}--\eqref{eq:mu*def}, using an equilibrium generating function $\phi$, where $\phi=(\phi^i_t)_{i\in\cN,t\in\mathcal{T}}$, $\phi^i_t : \times_{i\in\cN} \Delta(\cX^i) \to \left\{\cX^i \to \Delta(\cA^i) \right\}$, common belief update function $\underline{F}$ and prior distributions $Q_1$. As a consequence, $\beta^{*,i}_t$ only depends on current type $x_t^i$ of player $i$, and on the common information $a_{1:t-1}$ through the set of marginals $\underline{\mu}^*_t[a_{1:t-1}]$, and $\mu_t^{*,i}$ depends only on common information history $a_{1:t-1}$.

\begin{theorem}[Converse]
	\label{thm:2}
	Let ($\beta^*,\mu^*$) be an SPBE. Then there exists an equilibrium generating function $\phi$
 that satisfies \eqref{eq:m_FP} in backward recursion $\forall\ \pi_t = \mu^*_t[a_{1:t-1}], \ \forall\ a_{1:t-1}$,
	such that  ($\beta^*,\mu^*$) is defined through forward recursion using $\phi$.\footnote{Note that for $\underline{\pi}_t \neq \underline{\mu}^*_t[a_{1:t-1}] $ for any $a_{1:t-1}$, $\phi$ can be arbitrarily defined without affecting the definition of $(\beta^*,\mu^*)$.}
\end{theorem}
\begin{IEEEproof}
	Please see Appendix~\ref{app:id}.
\end{IEEEproof}

\vspace{-0.35cm}

\subsection{Discussion}\label{sec:discussion}

Several remarks are in order with regard to the above methodology and the result.

\emph{Remark 1: }
The second sub-case in~\eqref{eq:F_update} dictates how beliefs are updated for histories with zero probability. The particular expression used is only one of many possible updates than can be used here. Dynamics that govern the evolution of public beliefs at histories with zero probability of occurrence affect equilibrium strategies. Thus, the construction proposed for calculating PBEs in this paper will produce a different set of equilibria if one changes the second sub-case above. The most well-known example of another such update is the \emph{intuitive criterion} proposed in~\cite{chokreps87} for Nash equilibria, later generalized to sequential equilibria in~\cite{cho87}. The intuitive criterion assigns zero probability to states that can be excluded based on data available to all players (in our case action profile history $ a_{1:t-1} $). Another example of belief update is \emph{universal divinity}, proposed in~\cite{sobel87}.

\emph{Remark 2:} To highlight the significance of the unique structure of \eqref{eq:m_FP}, one can think as follows.
When all players other than player $i$ play structured strategies, i.e., strategies of the form $A^j_t\sim m^j_t(\cdot|\upi_t,x^j_t)=\theta^j_t[\upi_t](\cdot|x^j_t)$, one may want to characterize the optimization problem from the viewpoint of the $i$-th player in order to characterize its best response.
In particular one may want to show that although player $i$ can play general strategies of the form $A^i_t\sim g^i_t(\cdot|x^i_{1:t},a_{1:t-1})$, it is sufficient to best respond with structured strategies of the form
$A^i_t\sim m^i_t(\cdot|\upi_t,x^i_t)=\theta^i_t[\upi_t](\cdot|x^i_t)$
as well.
To show that, one may entertain the thought that player $i$ faces an MDP with state $(X^i_t,\uPi_t)$, and action $A^i_t$ at time $t$.
If that were true, then player $i$'s optimal action could be characterized (using standard MDP results)  by a dynamic-programming equation similar to~\eqref{eq:m_FP}, of the form
 \eq{
 \tilde{\gamma}^{i}_t(\cdot|x_t^i) \in
 \arg\max_{\gamma^i_t(\cdot|x_t^i)} \E^{\gamma^i_t(\cdot|x_t^i) \tilde{\gamma}^{-i}_t,\,\pi_t} \left\{ R_t^i(X_t,A_t) + V_{t+1}^i (\uF(\underline{\pi}_t, \gamma^i_t(\cdot|x^i_t), \tilde{\gamma}^i_t(\cdot|\cdot), \tilde{\gamma}^{-i}_t, A_t), X_{t+1}^i) \big\lvert x_t^i \right\}, \label{eq:alternative}
  }
where, unlike~\eqref{eq:m_FP}, in the belief update equation the partial strategy $\gamma^i_t(\cdot|x^i_t)$ is also optimized over.
However, as it turns out, user $i$ does not face such an MDP problem!
The reason is that the update equation $\upi_{t+1}=\uF(\upi_t,\gamma_t,a_t)$ also depends on $\gamma^i_t$ which is the partial strategy of player $i$ and this has not been fixed in the above setting.
If however the update equation is first fixed (so it is updated as $\upi_{t+1}=\uF(\underline{\pi}_t, \tilde{\gamma}^i_t, \tilde{\gamma}^{-i}_t, a_t)=\uF(\underline{\pi}_t, \theta_t[\pi_t], a_t)$, i.e., using the equilibrium strategies even for player $i$) then indeed the problem faced by user $i$ is the MDP defined above.
%
%
It is now clear why~\eqref{eq:m_FP} has the flavor of a fixed-point equation: the update of beliefs needs to be fixed beforehand with the equilibrium action $\tilde{\gamma}^i_t$ even for user $i$, and only then user $i$'s best response can depend only on the MDP state $(X^i_t,\uPi_t)$ thus being a structured strategy as well. This implies that his optimal action $\tilde{\gamma}^i_t$ appears both on the left and right hand side of this equation giving rise to~\eqref{eq:m_FP}.

\emph{Remark 3:} In this paper, we find a class of PBEs of the game, while there may exist other equilibria that are not ``structured", and can not be found by directly using the proposed methodology. The rationale for using structured equilibria over others is the same as that for using MPE over SPE for a symmetric information game; a focussing argument for using simpler strategies being one of them.

\section{A Methodology for SPBE computation in Infinite Horizon} \label{secih}
In this section we consider the infinite horizon discounted reward dynamic game $ \mathfrak{D}_{\infty} $.
%
%
%
We state the fixed-point equation that defines the value function and strategy mapping for the infinite horizon problem. This is analogous to the backwards recursion (\eqref{eq:m_FP} and~\eqref{eq:Vdef}) that define the value function and $ \theta $ mapping for the finite horizon problem.

Define the set of functions $ V^i: \times_{j=1}^N \Delta(\mX^j) \times \mX^i \rightarrow \mathbb{R} $ and strategies $ \tilde{\gamma}^i:\mX^i \rightarrow \Delta(\mA^i) $ (which are generated formally as $ \tilde{\gamma}^i = \theta^i[\upi] $ for given $ \upi $) via the following fixed-point equation: $ \forall $ $ i \in \mN $, $ x^i \in \mX^i $,
\begin{subequations} \label{eqihfpe}
\begin{align} \label{eqihfpeA}
	\tilde{\gamma}^i(\cdot \mid x^i) &\in \argmax_{\gamma^i(\cdot \mid x^i) \in \Delta(\mA^i)} \mE^{\gamma^i(\cdot \mid x^i),\tilde{\gamma}^{-i},\pi^{-i}} \lpr R^i(X,A) 
	+ \delta V^i\big(\uF(\upi,\tilde{\gamma},A),{X^\prime}^{i}\big) \mid \upi,x^i \rpr,
\\
	V^i(\upi,x^i) &= \mE^{\tilde{\gamma}^i(\cdot \mid x^i),\tilde{\gamma}^{-i},\pi^{-i}} \lpr R^i(X,A)
	+ \delta V^i\big(\uF(\upi,\tilde{\gamma},A),{X^\prime}^{i}\big) \mid \upi,x^i \rpr.
\end{align}
\end{subequations}
Note that the above is a joint fixed-point equation in $ (V,\tilde{\gamma}) $, unlike the backwards recursive algorithm earlier which required solving a fixed-point equation only in $ \tilde{\gamma} $. Here the unknown quantity is distributed as $ (X^{-i},A^i,A^{-i},{X^\prime}^{i}) \sim \pi^{-i}(x^{-i})\gamma^i(a^i \mid x^i) \tilde{\gamma}^{-i}(a^{-i} \mid x^{-i}) Q^i({x^\prime}^{i} \mid x^i,a) $, 
and $ F^i(\cdot) $ is defined in~\eqref{eq:F_update}.

Define the belief $ \mu^\star $ inductively similar to the forward recursion from Section~\ref{sec:fr}.
By construction the belief defined above satisfies the consistency condition needed for a PBE. Denote the strategy arising out of $ \tilde{\gamma} $ by $ \beta^\star $ i.e.,
\begin{gather} \label{eqbeta}
\beta_t^{i,\star}(a_t^i \mid x_{1:t}^i,a_{1:t-1}) = \theta^i\big[\underline{\mu}_t^{\star}[a_{1:t-1}]\big](a_t^i \mid x_t^i).
\end{gather}
Note that although the mapping $ \theta^i $ is stationary, the strategy $ \beta_t^{i,*} $ derived from it is not so. 
Below we state the central result of this section, that the strategy-belief pair $ (\beta^\star,\mu^\star) $ constructed from the solution of the fixed-point equation~\eqref{eqihfpe} and the forward recursion
indeed constitutes a PBE.

\begin{theorem}\label{thih}
	Assuming that the fixed-point equation~\eqref{eqihfpe} admits an absolutely bounded solution $ V^i $ (for all $ i \in \mN $), the strategy-belief pair $ (\beta^\star,\mu^\star) $ defined in~\eqref{eqbeta}
	is a PBE of the infinite horizon discounted reward dynamic game i.e., $ \forall $ $ i \in \mN $, $ \beta^i $, $ t \ge 1 $, $ h_t^i \in \mathcal{H}_t^i $,
	\eq{ 	
	\mE^{\beta^{i,\star},\beta^{-i,\star},\mu_t^\star[h_t^c]} \lpr \sum_{n=t}^\infty \delta^{n-t} R^i(X_n,A_n) \mid h_t^i \rpr
	\ge
	\mE^{\beta^{i},\beta^{-i,\star},\mu_t^\star[h_t^c]} \lpr \sum_{n=t}^\infty \delta^{n-t} R^i(X_n,A_n) \mid h_t^i \rpr.
	}	
\end{theorem}


\begin{IEEEproof}
	Please see Appendix~\ref{app:ih}.
\end{IEEEproof}

Our approach to proving Theorem~\ref{thih} is as follows. We begin by noting that the standard contraction mapping arguments used in infinite horizon discounted reward MDPs/POMDPs viewed as a limit of finite horizon problems, do not apply here, since the policy equation~\eqref{eqihfpeA} is not a maximization, but a different fixed-point equation. So we attempt to ``fit" the infinite horizon problem into the framework of finite-horizon model developed in the previous section. We do that by first introducing a terminal reward that depends on common beliefs, in the backward-forward recursion construction of Section~\ref{sec:fh} for finite horizon games.
We consider a finite horizon, $ T > 1 $, dynamic game with rewards same as in the infinite horizon version and time invariant transition kernels $ Q^i $. For each player $ i $, there is a terminal reward $ G^i(\pi_{T+1},x_{T+1}^i) $ that depends on the terminal type of player $ i $ and the terminal belief. It is assumed that $ G^i(\cdot) $ is absolutely bounded. We define the value functions $ \big(V_t^{i,T}: \times_{j \in \mN} \Delta(\mX^j) \times \mX^i \rightarrow \mathbb{R} \big)_{i \in \mN,t\in \mathcal{T} } $ and strategies $ \big( \tilde{\gamma}_t^{i,T} \big)_{i \in \mN,t\in \mathcal{T} } $ backwards inductively in the same way as in Section~\ref{sec:fhbr} except Step 1, where instead of~\eqref{eq:VT+1} we set $ V_{T+1}^{i,T} \equiv G^i  $. This consequently results in a strategy/belief pair $ (\beta^\star,\mu^\star) $, based on the forward recursion in Section~\ref{sec:fr}.
Now, due to the above construction, the value function $ V_t^{i,T} $ from above and $ V^i $ from~\eqref{eqihfpe} are related (please see Lemma~\ref{lemfhtoih} in Appendix~\ref{app:interm_ih}). %
This result combined with continuity arguments as $ T \to \infty $ complete the proof of Theorem~\ref{thih}.

\section{An Existence Result for the Fixed-Point Equation} \label{sec:existence}

In this section, we discuss the problem of existence of signaling equilibria.\footnote{In the special case of uncontrolled types where player $i$'s instantaneous reward does not depend on its private type $x_t^i$, the fixed point equation always has a type-independent, myopic solution $\tilde{\gamma}^i_t(\cdot)$, since it degenerates to a best-response-like equation similar to the one for computing Nash equilibrium. This result is shown in~\cite{OuTaTe15}.}
While it is known that for any finite dynamic game with asymmetric information and perfect recall, there always exists a PBE~\cite[Prop. 249.1]{OsRu94}, existence of SPBE is not guaranteed.
It is clear from our algorithm that existence of SPBE boils down to existence of a solution to the fixed-point equation~\eqref{eq:m_FP} in finite horizon and~\eqref{eqihfpe} in infinite horizon. Specifically, for the finite horizon, at each time $ t $ given the functions $ V_{t+1}^i $ for all $ i \in \mN $ from the previous round (in the backwards recursion) equation~\eqref{eq:m_FP} must have a solution $ \tilde{\gamma}_t^i $ for all $ i \in \mN $. Generally, existence of equilibria is shown through Kakutani's fixed point theorem, as is done in proving existence of a mixed strategy Nash equilibrium of a finite game~\cite{OsRu94,Na51}. This is done by showing existence of fixed point of the best-response correspondences of the game. Among other conditions, it requires the ``closed graph'' property of the correspondences, which is usually implied by the continuity property of the utility functions involved.
For~\eqref{eq:m_FP} establishing existence is not straightforward due to: (a) potential discontinuity of the $\pi_t$ update function $F$ when the denominator in the Bayesian update is 0 and (b) potential discontinuity of  the value functions, $V_{t+1}^i$.
In the following we provide sufficient conditions that can be checked at each time $ t $ to establish the existence of a solution.

We consider a generic fixed-point equation similar to the one encountered in Section~\ref{sec:fh} and Section~\ref{secih} and state conditions under which they are guaranteed to have a solution. To concentrate on the essential aspects of the problem we consider a simple case with $ N=2 $, type sets $ \mX^i = \{x^H,x^L\} $ and action sets $ \mA^i = \{0,1\} $. Furthermore, types are static and instantaneous rewards $ R^i(x,a) $ do not depend on $ x^{-i} $.

Given public belief $ \upi = (\pi^1,\pi^2) \in \times_{i=1}^2 \Delta(A^i) $, value functions $ V^1,V^2 $, one wishes to solve the following system of equations for $ \big( \tilde{\gamma}^i(\cdot \mid x^i) \big)_{{x^i \in \{x^H,x^L\}, i \in \{1,2\}}} $.
\begin{multline}
	\tilde{\gamma}^i(\cdot \mid x^i) \in \argmax_{\gamma^i(\cdot \mid x^i) \in \Delta(\mA^i)} \mE^{\gamma^i(\cdot \mid x^i),\tilde{\gamma}^{-i}} \Big\{ R^i(x^i,A)
	+ V^i\Big(\big(F^1(\pi^1,\tilde{\gamma}^1,A^1),F^2(\pi^2,\tilde{\gamma}^2,A^2)\big),x^i\Big) \mid x^i,\upi \Big\}
\end{multline}
where the expectation is evaluated using the probability distribution on $ (A^1,A^2) $,
\begin{gather}
\gamma^i(a^i \mid x^i) \big[ \pi^{j}(x^H) \tilde{\gamma}^{j}(a^j \mid x^H) + \pi^{j}(x^L) \tilde{\gamma}^{j}(a^j \mid x^L) \big].
\end{gather}

The probabilistic policy $ \tilde{\gamma} $ can be represented by the 4-tuple $ \underline{p} = \big( \tilde{p}^{1L},\tilde{p}^{2L},\tilde{p}^{1H},\tilde{p}^{2H} \big) $ where $ \tilde{p}^{iH} = \gamma^i(a^i=1 \mid x^H) $ and $ \tilde{p}^{iL} = \gamma^i(a^i=1 \mid x^L) $, $i=1,2$.

The fixed-point equation of interest reduces to
\begin{multline} \label{eqfpe}
\tilde{p}^{1H} \in \argmax_{a \in [0,1]} a \Big[ \big( \pi^2 \tilde{p}^{2H} + (1-\pi^2) \tilde{p}^{2L} \big) 
\big( V^1(F_1(\pi^1,\tilde{p}^1),F_1(\pi^2,\tilde{p}^2),x^H) - V^1(F_0(\pi^1,\tilde{p}^1),F_1(\pi^2,\tilde{p}^2),x^H) \big)
\\
+\big( 1 - \pi^2 \tilde{p}^{2H} - (1-\pi^2) \tilde{p}^{2L} \big) \big( V^1(F_1(\pi^1,\tilde{p}^1),F_0(\pi^2,\tilde{p}^2),x^H) 
- V^1(F_0(\pi^1,\tilde{p}^1),F_0(\pi^2,\tilde{p}^2),x^H) \big)
\\
+\big( \pi^2 \tilde{p}^{2H} + (1-\pi^2) \tilde{p}^{2L} \big) \big( R^1(x^H,1,1) - R^1(x^H,0,1) \big)
\\
+\big( 1 - \pi^2 \tilde{p}^{2H} - (1-\pi^2) \tilde{p}^{2L} \big) \big( R^1(x^H,1,0) - R^1(x^H,0,0) \big) \Big]
\end{multline}
and three other similar equations for $ \tilde{p}^{1L},\tilde{p}^{2H},\tilde{p}^{2L} $. 
\begin{subequations}
	\begin{align}
	F_1(\pi,(p^H,p^L)) &\triangleq \frac{\pi p^H}{\pi p^H + \bar{\pi} p^L}
	\\
	F_0(\pi,(p^H,p^L)) &\triangleq \frac{\pi (1-p^H)}{\pi (1-p^H) + \overline{\pi} (1-p^L)}
	\end{align}
\end{subequations}
and in both definitions, if the denominator is $ 0 $ then the RHS is taken as $ \pi $.
\vspace{-0.4cm}
\subsection{Points of Discontinuities and the Closed graph result}   \label{sec:existdisc}
Equation~\eqref{eqfpe} and the other three similar equations  are essentially of the form (for a given $ \upi $)
	\eq{
	\label{eqpp4}
	x \in \argmax_{a \in [0,1]} \ a f_1(x,y,w,z), \;\;\;
	y \in \argmax_{b \in [0,1]} \ b f_2(x,y,w,z) \nn\\
	w \in \argmax_{c \in [0,1]} \ c f_3(x,y,w,z), \;\;\;
	z \in \argmax_{d \in [0,1]} \ d f_4(x,y,w,z)
	}
with $ x,y,z,w $ as $ \tilde{p}^{1H},\tilde{p}^{1L},\tilde{p}^{2H},\tilde{p}^{2L} $, respectively.

Define $ \mathcal{D}_i \subseteq [0,1]^4 $ as the set of discontinuity points of $ f_i $ and $ \mathcal{D} \triangleq \cup_{i=1}^4 \mathcal{D}_i $.

For any point $ \ux_0 \in \mathcal{D} $, define $ S(\ux_0) $ as the subset of indexes $ i \in \{1,2,3,4\} $ for which $ f_i(\ux) $ is discontinuous at $ \ux_0 $.

\paragraph*{Assumption (E1)} At any point $ \ux_0 \in \mathcal{D} $, $ \forall $ $ i \in S(\ux_0) $ one of the following is satisfied:
\begin{enumerate}
	\item $ f_i(\ux_0) $ = 0, or
	
	\item $ \exists $ $ \epsilon > 0 $ such that $ \forall $ $ \ux \in B_{\epsilon}(\ux_0) $ (inside an $\epsilon$-ball of $\ux_0$) the sign of $ f_i(\ux) $ is same as the sign of $ f_i(\ux_0) $.
\end{enumerate}

In the following we provide a sufficient condition for existence.
\begin{theorem} \label{thm:exist}
	Under Assumption (E1), there exists a solution to the fixed-point equation~\eqref{eqpp4}.
\end{theorem}	
\begin{IEEEproof}
	Please see Appendix~\ref{app:existence}.
\end{IEEEproof}	



The above set of results provide us with an analytical tool for establishing existence of a solution to the concerned fixed-point equation.


While the above analytical result is useful in understanding a theoretical basis for existence, it doesn't cover all instances. For instance, fixed-point equation arising out of~\eqref{eq:m_FP} for $ t=1 $ from Section~\ref{sec:Example_fh}, does not satisfy assumption~(E1). In the following we provide a more computationally orientated approach to establishing existence and/or solving the generic fixed-point equation~\eqref{eqpp4}.

We motivate this case-by-case approach with the help of an example. Suppose we hypothesize that the solution to~\eqref{eqpp4} is such that $ x=0,w=0 $ and $ y,z \in (0,1) $. Then~\eqref{eqpp4} effectively reduces to  checking if there exists $ y^\star,z^\star \in (0,1) $ such that
\begin{subequations}
	\begin{gather}
	\label{eqss1}
	y^\star \in \argmax_b~ b \, f_2(0,y^\star,0,z^\star) \;\;\;
	z^\star \in \argmax_d~ d \, f_4(0,y^\star,0,z^\star)\\	
	\label{eqss-1}
	f_1(0,y^\star,0,z^\star) \le 0 \;\;\;\;\;
	f_3(0,y^\star,0,z^\star) \le 0.
	\end{gather}
\end{subequations}
Thus the 4-variable system reduces to solving a 2-variable system and 2 conditions to verify.
For instance, if $f_2(0,y,0,z)$, $f_4(0,y,0,z)$ as functions of $y,z$ satisfy the conditions of Theorem~\ref{thm:exist}
then the sub-system~\eqref{eqss1} has a solution. If one of these solution is also consistent with~\eqref{eqss-1} then this sub-case indeed provides a solution to~\eqref{eqpp4}.

Generalizing the simplification provided in the above example, we divide solutions into $ 3^4=81 $ cases\footnote{Generally, the number of cases is $ 3^{\sum_{i=1}^N M_i} $ where $ N $ is the number of agents and $ M_i $ is the number of types for player $ i $.} based on whether each of $ x,y,w,z $ are in  $ \{0\},(0,1),\{1\} $. There are
(1) 16 corner cases where none are in the strict interior $ (0,1) $;
(2) 32 cases where exactly one is in the strict interior $ (0,1) $;
(3) 24 cases where 2 variables  are in the strict interior $ (0,1) $;
(4) 8 cases where 3 variables are in the strict interior $ (0,1) $;
(5) 1 case where all 4 variables are in the strict interior  $ (0,1) $.

Similar to the calculations above, for each of the 81 cases one can write a sub-system to which the problem~\eqref{eqpp4} effectively reduces to. Clearly, if any one of the 81 sub-systems has a solution then the problem~\eqref{eqpp4} has a solution. Furthermore, searching for a solution reduces to an appropriate sub-problem depending on the case.

The approach then is to enumerate each of these 81 cases (as stated above) and check them in order. However this case-by-case division provides a computational simplification - not all cases require solving the entire fixed-point equation. Whenever a variable, say $ y $, is not in the strict interior $ (0,1) $ then the corresponding equation~\eqref{eqpp4} need not be solved, since one only needs to verify the sign at a specific point. Hence, all sub-cases of (1) reduce to simply checking the value of functions $ f_i $ at corner points - no need for solving a fixed-point equation. All sub-cases of (2) reduce to solving a 1-variable fixed-point equation and three corresponding conditions to verify, etc.
\vspace{-0.2cm}

\section{A Concrete Example of Multi-stage investment in Public Goods} \label{secexample} \label{sec:Example}
Here we discuss both, a two-stage (finite) and an infinite-horizon version of a public goods example to illustrate the methodology described above for the construction of SPBE.
\vspace{-0.4cm}
\subsection{A two stage public goods game} \label{sec:Example_fh}
We consider a discrete version of Example~8.3 from \cite[ch.8]{FuTi91book}, which is an instance of a repeated public goods game. There are two players who play a two-period game. In each period $t$, they simultaneously decide whether to contribute to the period $t$ public good, which is a binary decision $a_t^i \in \{ 0,1\}$ for players $i=1,2$. Before the start of period 2, both players know the action profile from period 1. In each period, each player gets a reward of 1 if at least one player contributed and 0 if none contributed. Player $i$'s cost of contributing is $x^i$ which is its private information.
Both players believe that $x^{i}$'s are drawn independently and identically with probability distribution $Q$ with support $\{{x}^L, {x}^H\}$; $0<x^L<1<x^H$, and $\mathbb{P}^Q(X^i=x^H)=q \in (0,1)$. 

In our model this corresponds to $N=2, T=2$ and reward for player $i$ in period $t$ is $R_t^i(x,a_t) = \delta^t R^i(x,a_t)$, with $ R^i(x,a_t) = (1-x^i) \mathds{1}(a_t^i = 1) + a_t^{-i} \mathds{1}(a_t^i=0) $.
We set $ \delta=1 $ in this two-stage case.
%
%
We use the backward recursive algorithm from Section~\ref{sec:Result} to find an SPBE of this game.
Here the partial functions $\gamma_t^i$ can equivalently be defined through the scalars $ p^{iL}_t$, $p^{iH}_t \in [0,1] $, for $ t=1,2 $ and $ i=1,2 $, where
\begin{subequations}
\begin{align}	
\gamma_t^i(1|x^L) = p^{iL}_t, \quad
\gamma_t^i(0|x^L) = 1-p^{iL}_t,
\\
\gamma_t^i(1|x^H) = p^{iH}_t, \quad
\gamma_t^i(0|x^H) = 1-p^{iH}_t,
\end{align}
\end{subequations}
Henceforth, $p^{iL}_t, p^{iH}_t$ is used interchangeably with $\gamma_t^i$.

For $t=2$ and for any fixed $\underline{\pi}_2 = (\pi_2^1,\pi_2^2)$, where $\pi_2^i  = \pi^i_2(x^H)\in [0,1]$ represents a probability measure on the event $\{X^i = x^H \}$.  
Let $\tilde{\gamma}_2 = (\tilde{p}^{1L}_2, \tilde{p}^{2L}_2,\tilde{p}^{1H}_2,\tilde{p}^{2H}_2) = \theta_2[\underline{\pi}_2]$
be defined through the fixed point equation~\eqref{eq:m_FP}.
Since $1-x^H<0$, $ \tilde{p}^{iH}_2 = 0$ is the solution.
Thus the fixed-point equation can be reduced to, $\forall i \in \{1,2 \}$,
\eq{
	\hs{-0.1cm}\tilde{p}^{iL}_2 \in \arg\max_{{p}^{iL}_2}& (1-p_2^{iL})(1-\pi_2^{-i})\tilde{p}_2^{-iL} + p_2^{iL}(1-x^L). \label{eq:ext2}
}
This implies~\eqref{eq:p2eq} below, the solutions to which are shown in Figure~\ref{fig1} in the space of $(\pi_2^1,\pi_2^2)$.
\eq{
	\tilde{p}^{iL}_2 &= \lb{0 \;\;\;\;\;\;\;\;\;\;\;\;\;\; \hfill \text{if}\;\;\;\; \hfill x^L > 1-(1-\pi_2^{-i})\tilde{p}_2^{-iL},\\
		1 \;\;\;\;\;\;\;\;\;\;\;\;\;\;\;\; \hfill \text{if}\;\;\;\;\;\; \hfill x^L < 1-(1-\pi_2^{-i})\tilde{p}_2^{-iL},\\
		\text{arbitrary} \;\;\; \hfill \text{if}\;\;\;\; \hfill x^L = 1-(1-\pi_2^{-i})\tilde{p}_2^{-iL}.\\
	} \label{eq:p2eq}
}
%
\begin{figure}[!htbp]
	\centering
	\vspace{-0.85cm}
	\includegraphics[height=3in]{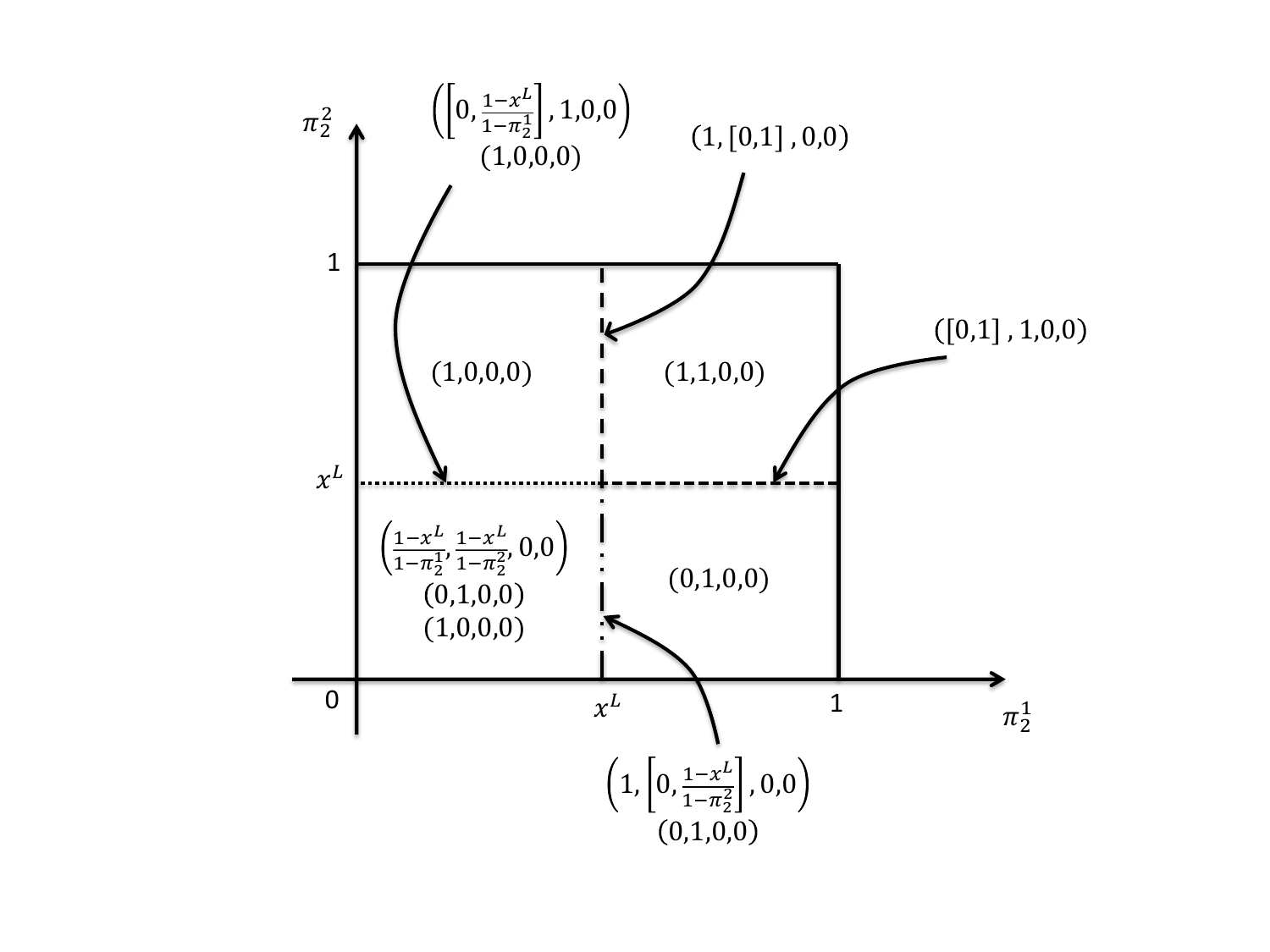}
	\vspace{-0.5in}
	\caption{Solutions of fixed point equation in \eqref{eq:p2eq}. Solutions are shown as quadruplets $(\tilde{p}^{1L}_2, \tilde{p}^{2L}_2,\tilde{p}^{1H}_2,\tilde{p}^{2H}_2)$ with intervals used whenever the solution is not uniquely defined.}
	\label{fig1}
	\vspace*{-1.5ex}
\end{figure}

Thus for any $\underline{\pi}_2$, there can exist multiple equilibria and correspondingly multiple $\theta_2[\underline{\pi}_2]$ can be defined. For any particular $\theta_2$, at $t=1$, the fixed point equation arising out of~\eqref{eq:m_FP} defines $\theta_1[Q^2]$, where $ Q^2 = Q \times Q $ denotes the profile of initial belief. 

Using one such $\theta_2$ defined below, we find an SPBE of the game for $q=0.1, x^L = 0.2, x^H=1.2$. We use $\theta_2[\underline{\pi}_2]$ as one possible set of solutions of \eqref{eq:p2eq}, 
described below,
\eq{
	\label{eq:theta2}
	\theta_2[\pi_2] =(\tilde{p}^{1L}_2,\tilde{p}^{2L}_2, \tilde{p}^{1H}_2,\tilde{p}^{2H}_2)
	= \lb{ (\frac{1-x^L}{1-\pi_2^1},\frac{1-x^L}{1-\pi_2^2},0,0) \;\;\;\;\hfill \pi_2^1 \in[0, x^L), \pi_2^2 \in [0,x^L) \text{ }\\
		(1,0,0,0) \;\;\;\hfill \pi_2^1 \in[0, x^L], \pi_2^2 \in [x^L,1] \text{ } \\
		(0,1,0,0)\;\;\;\hfill \pi_2^1 \in[x^L,1], \pi_2^2 \in[0, x^L] \text{ }\\
		(1,1,0,0)\;\;\;\hfill \pi_2^1\in(x^L , 1], \pi_2^2 \in (x^L , 1].
	}\hs{-0.2cm}
}

Then, through iteration on the fixed point equation 
and using the aforementioned $\theta_2[\underline{\pi}_2]$, we numerically find (and analytically verify) that $\theta_1[Q^2] =(\tilde{p}^{1L}_1,\tilde{p}^{2L}_1, \tilde{p}^{1H}_1,\tilde{p}^{2H}_1) = (0,1,0,0)$ is a fixed point. Thus
\seq{
	\eq{
		\beta^1_1(A_1^1 = 1|X^1=x^L) = 0 \hspace{10pt}&\hspace{10pt} \beta^2_1(A_1^2 = 1|X^2=x^L) = 1 \nonumber \\
		\beta^1_1(A_1^1 = 1|X^1=x^H) = 0 \hspace{10pt}&\hspace{10pt} \beta^2_1(A_1^2 = 1|X^2=x^H)= 0 \nonumber
	}
}
with beliefs $\underline{\mu}^*_2[00] = (q,1), \underline{\mu}^*_2[01] = (q,0), \underline{\mu}^*_2[10]=(q,1) , \underline{\mu}^*_2[11] =(q,0) $ and $\left(\beta^i_2(\cdot|a_1,\cdot)\right)_{i\in \{1,2\}}=\theta_2[\underline{\mu}^*_2[a_1]] $ is an SPBE of the game.
In this equilibrium, player 2 at time $t=1$, contributes according to her type whereas player 1 never contributes, thus player 2 reveals her private information through her action whereas player 1 does not.
Since $\theta_2$ is symmetric, there also exists an (antisymmetric) equilibrium where at time $t=1$, players' strategies reverse i.e. player 2 never contributes and player 1 contributes according to her type.
We can also obtain a symmetric equilibrium where $\theta_1[Q^2] = (\frac{1-x^L}{(1-q)(1+x^L)},\frac{1-x^L}{(1-q)(1+x^L)},0,0)$ as a fixed point when $x^L > \frac{q}{2-q}$, resulting in beliefs $\underline{\mu}^*_2[00] = (p,p), \underline{\mu}^*_2[01] = (p,0), \underline{\mu}^*_2[10]=(0,p) , \underline{\mu}^*_2[11] =(0,0) $ where $p=\frac{q(1+x^L)}{q(1+x^L) + (1-x^L)}$.

\subsection{Infinite horizon version}

For the infinite horizon version we consider three values $ \delta = 0,0.5,0.95 $ and solve the corresponding fixed point equation (arising out of~\eqref{eqihfpe})  numerically to calculate the mapping $ \theta $. 
The fixed-point equation is solved numerically by discretizing the $ \upi- $space $ [0,1]^2 $ and all solutions that we find are symmetric w.r.t. players i.e., $ \tilde{p}^{1L} $ for $ \upi=(\pi^1,\pi^2) $ is the same as $ \tilde{p}^{2L} $ for $ \upi^\prime = (\pi^2,\pi^1) $ and similarly for $ \tilde{p}^{1H},\tilde{p}^{2H} $.

For $ \delta = 0 $, the game is instantaneous and actually corresponds to the second round $ t=2 $ play in the finite horizon two-stage version above. Thus whenever player $ 1 $'s type is $ x^H $, it is instantaneously profitable not to contribute. This gives $ \tilde{p}^{1H} = 0 $, for all $ \upi $. Thus we only plot $ \tilde{p}^{1L} $; in Fig.~\ref{fig:del0_p1L} (this can be inferred from the discussion and Fig.~\ref{fig1} above).
Intuitively, with type $ x^L $ the only values of $ \upi $ for which player 1 would not wish to contribute is if he anticipates player $ 2 $'s type to be $ x^L $ with high probability and rely on player 2 to contribute. This is why for lower values of $ \pi^2 $ (i.e., player $ 2 $'s type likely to be $ x^L $) we see $ \tilde{p}^{1L} = 0 $ in Fig.~\ref{fig:del0_p1L}.

\begin{figure}[!htbp]
	\centering
	\includegraphics[width=0.5\textwidth]{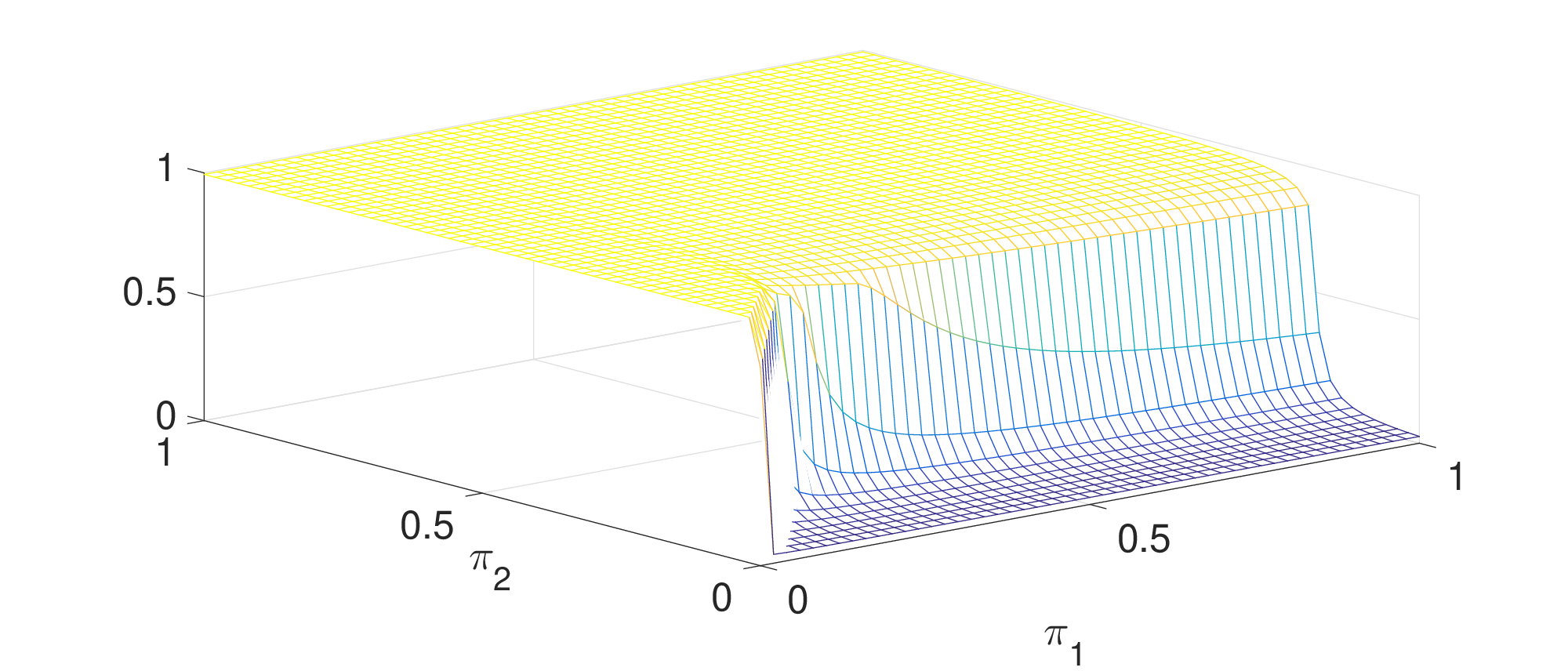}	
	\caption{$ \tilde{p}^{1L} $ vs. $ (\pi^1,\pi^2) $ at $ \delta = 0 $; $ (\tilde{p}^{1L},\tilde{p}^{1H}) = \theta^1[\pi^1,\pi^2] $.}	
	\label{fig:del0_p1L}	
\end{figure}

\begin{figure}[!htbp]
	\centering
	\includegraphics[width=\textwidth]{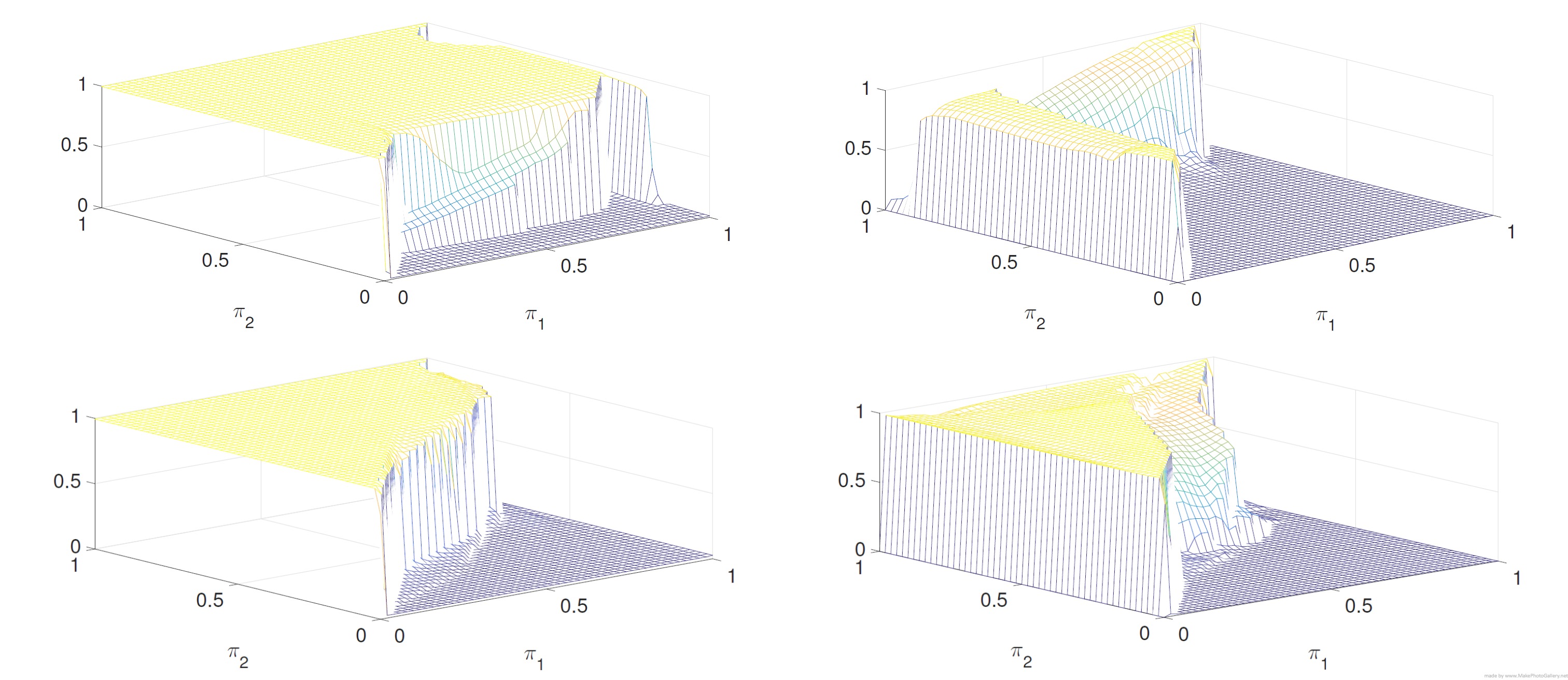}
	\caption{$ \tilde{p}^{1L},\tilde{p}^{1H} $ vs. $ (\pi^1,\pi^2) $ at $ \delta = 0.5 $ (upper left and right). $ \tilde{p}^{1L},\tilde{p}^{1H} $ vs. $ (\pi^1,\pi^2) $ at $ \delta = 0.95 $ (lower left and right).}
	\label{fig:combined}
\end{figure}

Now consider $ \tilde{p}^{1L} $ plotted in Fig.~\ref{fig:del0_p1L}
and~\ref{fig:combined}. As $ \delta $ increases, future rewards attain more priority and signaling comes into play. So while taking an action, players not only look for their instantaneous reward but also how their action affects the future public belief $ \pi $ about their private type. It is evident in the figures that as $ \delta $ increases, at high $ \pi^1 $, up to larger values of $ \pi^2 $ player $ 1 $ chooses not to contribute when his type is $ x^L $. This way he intends to send a ``wrong'' signal to player $ 2 $ i.e., that his type is $ x^H $ and subsequently force player $ 2 $ to invest. This way player 1 can free-ride on player $ 2 $'s investment.

Now consider $ \tilde{p}^{1H} $ plotted in Fig.~\ref{fig:combined}.  
For $ \delta=0 $ we know that not contributing is profitable, however as $ \delta $ increases from $ 0 $, players are mindful of future rewards and thus are willing to contribute at certain beliefs.
Specifically, coordination via signaling is evident here. Although it is instantaneously not profitable to contribute if player $ 1 $'s type is $ x^H $, by contributing at higher values of $ \pi^2 $ (i.e., player $ 2 $'s type is likely $ x^H $) and low $ \pi^1 $, player $ 1 $ coordinates with player $ 2 $ to achieve net profit greater than $ 0 $ (reward when no one contributes). This is possible since the loss when contributing is $ -0.2 $ whereas the profit from free-riding on player $ 2 $'s contribution is $ 1 $.

Under the equilibrium strategy, beliefs $ \underline{\Pi}_t $ form a Markov chain. One can trace this Markov chain to study the signaling effect at equilibrium. On numerically simulating this Markov chain for the above example (at $ \delta= 0.95 $) we observe that for almost all initial beliefs, within a few rounds players completely learn each other's private type truthfully (or at least with very high probability). In other words, players manage to reveal their private type via their actions at equilibrium and to such an extent that it negates any possibly incorrect initial belief about their type.

As a measure of cooperative coordination at equilibrium one can perform the following calculation. Compare the value function $ V^1(\cdot,x) $ of agent $ 1 $ arising out of the fixed-point equation, for $ \delta = 0.95 $ and $ x \in \{x^H,x^L\} $  (normalize it by multiplying with $ 1-\delta $ so that it represents per-round value) with the best possible attainable single-round reward under a symmetric mixed strategy with a) full coordination and b) no coordination. Note that the two cases need not be equilibrium themselves, which is why this will result in a bound on the efficiency of the evaluated equilibria.

In case a), assuming both agents have the same type $ x $, full coordination can lead to the best possible reward of $ \frac{1+1-x}{2} = 1- \frac{x}{2} $ i.e., agent $ 1 $ contributes with probability $ 0.5 $ and agent $ 2 $ contributes with probability $ 0.5 $ but in a coordinated manner so that it doesn't overlap with agent $ 1 $ contributing. 

In case b) when agents do not coordinate and invest with probability $ p $ each, then the expected single-round reward is $ p(1-x) + p(1-p) $. The maximum possible value of this expression is $ (1-\frac{x}{2})^2 $.

For $ x=x^L=0.2 $, the range of values of $ V^1(\pi_1,\pi_2,x^L) $ over $ (\pi_1,\pi_2) \in [0,1]^2 $ is $ [0.865,0.894] $. 
Whereas full coordination produces $ 0.9 $ and no coordination $ 0.81 $. It is thus evident that agents at equilibrium end up achieving reward close to the best possible and gain significantly compared to the strategy of no coordination.

Similarly for $ x = x^H = 1.2 $ the range is $ [0.3,0.395] $. Whereas full coordination produces $ 0.4 $ and no coordination $ 0.16 $. The gain via coordination is evident here too.

\section{Conclusion}
\label{sec:Conc}

%
%
%
%
%
%

In this paper we presented a methodology for evaluating SPBE for games with asymmetric information and independent private types evolving as controlled Markov processes. The main contribution is a time decomposition akin to dynamic programming.
This decomposition allows one to find SPBE that exhibit signaling behavior with linear complexity in the time horizon.
Using this methodology, dynamic LQG games with asymmetric information are studied in~\cite{VaAn16a} where it is shown that under certain conditions, there exists an SPBE of the game with strategies being linear in players' private types. In~\cite{VaAn16b}, authors extend the finite-horizon model in this paper such that players do not observe their own types, rather make independent noisy observations of their types. An analogous backward-forward algorithm is presented for that model.
It is worth noting that although structured strategies are useful in making the equilibrium finding process tractable, no claim can be made about whether the resulting equilibrium outcomes are better or worse than those corresponding to general strategies. We believe this is an interesting future research direction.
Another interesting future direction is dynamic mechanism design for asymmetric information systems. 

	\appendices
			\section*{Acknowledgment}
			The authors wish to acknowledge Vijay Subramanian for his contribution to the paper.
Achilleas Anastasopoulos wishes to acknowledge Ashutosh Nayyar for the fruitful discussion and criticism of an early draft of this work presented during the ITA 2012 conference.

\section{Proof of Lemma~\ref{fact:G2S}}
	\label{app:A1}	
We prove this Lemma in the following steps.

\begin{itemize}
\item[(a)] In Claim~\ref{claim:CondInd}, we prove that for any policy profile $g$ and $\forall t \in \mathcal{T}$, $x_{1:t}^{i}$ for $i\in \cN$ are conditionally independent given the common information $a_{1:t}$.
 \item[(b)] In Claim~\ref{claim:B2}, using Claim~\ref{claim:CondInd}, we prove that for every fixed strategy $g^{-i}$ of the players $-i$, $((A_{1:t-1},X_t^i), A_t^i )_{t\in \mathcal{T}}$ is a controlled Markov process for player $i$.
\item[(c)] For a given policy $g$, we define a policy $s^i$ of player $i$ from $g$ as $s^i_t(a_t^i|a_{1:t-1},x_t^i) \defeq \mP^g(a_t^i|a_{1:t-1},x_t^i)$.
\item[(d)] In Claim~\ref{claim:B3}, we prove that the dynamics of this controlled Markov process
 $ \big((A_{1:t-1}, X_t^i ), A_t^i\big)_{t\in \mathcal{T}}$ under $(s^ig^{-i})$ are same as under $g$ i.e. $ \mP^{s^ig^{-i}} (x_t^i ,x_{t+1}^i, a_{1:t}) = \mP^g(x_t^i, x_{t+1}^i, a_{1:t} )$.
\item[(e)] In Claim~\ref{claim:B4}, we prove that w.r.t. random variables $(x_t,a_t)$, $x_t^i$ is sufficient for player $i$'s private information history $x_{1:t}^{i}$  i.e. $ \mP^g(x_t,a_t|a_{1:t-1},x_{1:t}^{i},a_t^i) = \mP^{g^{-i}}(x_t,a_t|a_{1:t-1},x_t^i,a_t^i)$.
\item[(f)] From (c), (d) and (e) we then prove the result of the lemma that $ \mP^{s^ig^{-i}}(x_t, a_t) = \mP^g(x_t, a_t)$.\\
\end{itemize}

\begin{claim}
	For any policy profile $g$ and $\forall t$,
	\eq{
	\mP^g(x_{1:t}|a_{1:t-1}) = \prod_{i=1}^N \mP^{g^i}(x_{1:t}^{i}|a_{1:t-1})
	}
	\label{claim:CondInd}
	\end{claim}
	\begin{IEEEproof}
	\seq{
	\eq{
	\mP^g(x_{1:t}|a_{1:t-1})
	&= \frac{\mP^g(x_{1:t},a_{1:t-1})}{\sum_{\bar{x}_{1:t}} \mP^g(\bar{x}_{1:t},a_{1:t-1})} \\
	&= \frac{\prod_{i=1}^N \left(Q_1^i(x^i_1)g^i_1(a_1^i|x_{1}^{i})\prod_{n=2}^t Q_n^i(x^i_{n}|x^i_{n-1},a_{n-1}) g^i_n(a_n^i|a_{1:n-1},x_{1:n}^{i}) \right)}{\sum_{\bar{x}_{1:t}} \prod_{i=1}^N \left(Q^i(\bar{x}^i_1)g^i_1(a_1^i|\bar{x}_{1}^{i})\prod_{n=2}^t Q_n^i(\bar{x}^i_{n}|\bar{x}^i_{n-1},a_{n-1}) g^i_n(a_n^i|a_{1:n-1},\bar{x}_{1:n}^{i}) \right)}\\
	&= \frac{\prod_{i=1}^N \left(Q_1^i(x^i_1)g^i_1(a_1^i|x_{1}^{i})\prod_{n=2}^t Q_n^i(x^i_{n}|x^i_{n-1},a_{n-1}) g^i_n(a_n^i|a_{1:n-1},x_{1:n}^{i}) \right)}{ \prod_{i=1}^N \left(\sum_{\bar{x}_{1:t}^{i}}Q^i(\bar{x}^i_1)g^i_1(a_1^i|\bar{x}_{1}^{i})\prod_{n=2}^t Q_n^i(\bar{x}^i_{n}|\bar{x}^i_{n-1}, a_{n-1}) g^i_n(a_n^i|a_{1:n-1},\bar{x}_{1:n}^{i})\right)}\\
	&= \prod_{i=1}^N\frac{ Q_1^i(x^i_1)g^i_1(a_1^i|x_{1}^{i})\prod_{n=2}^t Q_n^i(x^i_{n}|x^i_{n-1},a_{n-1}) g^i_n(a_n^i|a_{1:n-1},x_{1:n}^{i}) }{\sum_{\bar{x}_{1:t}^{i}}Q^i(\bar{x}^i_1)g^i_1(a_1^i|\bar{x}_{1}^{i})\prod_{n=2}^t Q_n^i(\bar{x}^i_{n}|\bar{x}^i_{n-1},a_{n-1}) g^i_n(a_n^i|a_{1:n-1},\bar{x}_{1:n}^{i})}\\
	&= \prod_{i=1}^N \mP^{g^i}(x_{1:t}^{i}|a_{1:t-1})
	}
	}
	\end{IEEEproof}
		\begin{claim}
	For a fixed $g^{-i}$, $\{(A_{1:t-1},X_t^i ), A_t^i \}_t$ is a controlled Markov process with state $(A_{1:t-1},X_t^i )$ and control action $A_t^i$.
	\label{claim:B2}
	\end{claim}
	\begin{IEEEproof}
	\seq{
	\eq{
	&\mP^g(\tilde{a}_{1:t},x_{t+1}^i|a_{1:t-1}, x_{1:t}^{i}, a_{1:t}^{i}) \nonumber \\
	&= \sum_{x_{1:t}^{-i}} \mP^g(\tilde{a}_{1:t}, x_{t+1}^i ,x_{1:t}^{-i}|a_{1:t-1},x_{1:t}^{i},a^i_t )\\
	&= \sum_{  x_{1:t}^{-i} } \mP^g(\tilde{a}_t^{-i},x_{t+1}^i ,x_{1:t}^{-i} |a_{1:t-1},x_{1:t}^{i},a^i_t ) I_{(a_{1:t-1},a^i_t)}(\tilde{a}_{1:t-1},\tilde{a}_t^i)\\
	&= \sum_{x_{1:t}^{-i} } \mP^{g^{-i}}(x_{1:t}^{-i}|a_{1:t-1}) \left( \prod_{j\neq i} g^j_t(\tilde{a}_t^j|a_{1:t-1},x_{1:t}^{j})\right) Q_t^i(x_{t+1}^i|x_t^i ,a^i_t,\tilde{a}^{-i}_t) I_{(a_{1:t-1},a^i_t)}(\tilde{a}_{1:t-1},\tilde{a}_t^i) \label{eq:XiiCmp1} \\
	&= \mP^{g^{-i}}(\tilde{a}_{1:t},x_{t+1}^i| a_{1:t-1},x^{i}_t, a^i_t),
	}
	}
where \eqref{eq:XiiCmp1} follows from Claim~\ref{claim:CondInd} since $x_{1:t}^{-i}$ is conditionally independent of $x_{1:t}^{i}$ given $a_{1:t-1}$ and the corresponding probability is only a function of $g^{-i}$.
	\end{IEEEproof}

 For any given policy profile $g$, we construct a policy $s^i$ in the following way,
\seq{\label{eq:defSi_main}
	\eq{ s^i_t(a_t^i|a_{1:t-1},x_t^i) &\defeq \mP^g(a_t^i|a_{1:t-1},x_t^i) \\
	&= \frac{\sum_{x_{1:t-1}^{i}} \mP^g(a_t^i ,x_{1:t}^{i}|a_{1:t-1})}{\sum_{\tilde{a}^i_t}\sum_{\tilde{x}_{1:t-1}^{i}} \mP^g(\tilde{a}_t^i, \tilde{x}_{1:t-1}^{i}x_t^i |a_{1:t-1})}\\
	&=\frac{ \sum_{x_{1:t-1}^{i}}\mP^{g^i}(x_{1:t}^{i}|a_{1:t-1}) g^i_t(a_t^i|a_{1:t-1},x_{1:t}^{i}) }{\sum_{\tilde{a}^i_t} \sum_{\tilde{x}_{1:t-1}^{i}}\mP^{g^i}(\tilde{x}_{1:t-1}^{i}x_t^i|a_{1:t-1})  g_t^i(\tilde{a}_t^i|a_{1:t-1},\tilde{x}_{1:t-1}^{i}x_t^i)}\label{eq:Sdef1}\\
	&=\mP^{g^i}(a_t^i|a_{1:t-1},x_t^i) \label{eq:defSi},
	}
	}
	where dependence of \eqref{eq:Sdef1} on only $g^i$ is due to Claim~\ref{claim:CondInd}.
	\begin{claim}
	The dynamics of the Markov process $\{(A_{1:t-1},X_t^i ), A_t^i \}_t$ under $(s^ig^{-i})$ are the same as under $g$ i.e.,
	\eq{
	\mP^{s^ig^{-i}} (x_t^i  ,x_{t+1}^i, a_{1:t} ) = \mP^g(x_t^i, x_{t+1}^i ,a_{1:t} ) \;\;\;\; \forall t \label{eq:D0}
	}
	\label{claim:B3}
	\end{claim}
	\begin{IEEEproof}
	We prove this by induction. Clearly,
	\eq{
	\mP^g(x_{1}^i) &= \mP^{s^ig^{-i}}(x^i_{1}) = Q_1^i(x^i_1).\label{eq:D1}
	}
	Now suppose \eqref{eq:D0} is true for $t-1$ which also implies that the marginals $\mP^g(x_{t}^i,a_{1:t-1}) = \mP^{s^ig^{-i}}(x^i_{t},a_{1:t-1})$. Then
	\seq{
	\eq{
	\mP^g(x_t^i ,a_{1:t-1}, x_{t+1}^i,a_t ) &= \mP^g(x^i_{t},a_{1:t-1})\mP^g(a^i_t| a_{1:t-1}, x_t^i )\mP^g(x_{t+1}^i,a_{1:t}| x_t^i, a_{1:t-1}, a^i_t )\\
	&= \mP^{s^ig^{-i}}(x^i_{t},a_{1:t-1})s^i_t(a^i_t|a_{1:t-1},x_t^i )\mP^{g^{-i}}(x_{t+1}^i,a_{1:t}| x_t^i, a_{1:t-1}, a^i_t )\label{eq:D2}\\
	&= \mP^{s^ig^{-i}}(x_t^i ,a_{1:t-1} , x_{t+1}^i,a_t ),
	}
	}
	where \eqref{eq:D2} is true from induction hypothesis, definition of $s^i$ in \eqref{eq:defSi} and since $\{(a_{1:t-1},x_t^i ), a_t^i \}_t$ is a controlled Markov process as proved in Claim~\ref{claim:B2} and its update kernel does not depend on policy $g^i$.This completes the induction step.
	
	\end{IEEEproof}
	
	\begin{claim}
	For any policy $g$,
	\eq{
	\mP^g(\tilde{x}_t ,\tilde{a}_t|a_{1:t-1},x_{1:t}^{i},a_t^i) = \mP^{g^{-i}}(\tilde{x}_t,\tilde{a}_t|a_{1:t-1},x_t^i,a_t^i). \label{eq:D5}
	}
	\label{claim:B4}
	\end{claim}
	\begin{IEEEproof}
	\eq{
	\mP^g(\tilde{x}_t ,\tilde{a}_t|a_{1:t-1},x_{1:t}^{i},a_t^i) = I_{x_{t}^{i},a^i_t}(\tilde{x}^i_t,\tilde{a}^i_t) \mP^g(\tilde{x}^{-i}_t ,\tilde{a}^{-i}_t|a_{1:t-1},x_{1:t}^{i}).
		}
	Now
	\seq{
	\eq{
	\mP^g(\tilde{x}^{-i}_t, \tilde{a}^{-i}_t|a_{1:t-1},x_{1:t}^{i}) &= \sum_{\tilde{x}_{1:t-1}^{-i}} \mP^g( \tilde{x}_{1:t}^{-i}, \tilde{a}^{-i}_t|a_{1:t-1},x_{1:t}^{i})\\
	&=\sum_{\tilde{x}_{1:t-1}^{-i}} \mP^g( \tilde{x}_{1:t}^{-i}|a_{1:t-1},x_{1:t}^{i}) \left( \prod_{j\neq i} g^j_t(\tilde{a}_t^j|a_{1:t-1},\tilde{x}_{1:t}^{j})\right)\\
	&=\sum_{\tilde{x}_{1:t}^{-i}} \mP^{g^{-i}}( \tilde{x}_{1:t}^{-i}|a_{1:t-1}) \left( \prod_{j\neq i} g^j_t(\tilde{a}_t^j|a_{1:t-1},\tilde{x}_{1:t}^{j})\right)\label{eq:D3}\\
	&=\mP^{g^{-i}}(\tilde{x}^{-i}_t, \tilde{a}^{-i}_t|a_{1:t-1})
	}
	}
	where \eqref{eq:D3} follows from Claim~\ref{claim:CondInd}.
	
	Hence
	\seq{
	\eq{
	\mP^g(\tilde{x}_t, \tilde{a}_t|a_{1:t-1},x_{1:t}^{i},a_t^i) &=I_{x_{t}^{i},a^i_t}(\tilde{x}^i_t,\tilde{a}^i_t) \mP^{g^{-i}}(\tilde{x}^{-i}_t ,\tilde{a}^{-i}_t|a_{1:t-1})\\
	&=\mP^{g^{-i}}(\tilde{x}_t, \tilde{a}_t|a_{1:t-1},x_t^i,a_t^i)
	}
	}
		\end{IEEEproof}
	Finally,
	\seq{
	\eq{\mP^g(\tilde{x}_t, \tilde{a}_t) &= \sum_{a_{1:t-1}x_{1:t}^{i}a_t^i}\mP^g(\tilde{x}_t,\tilde{a}_t|a_{1:t-1},x_{1:t}^{i},a_t^i) \mP^g(a_{1:t-1},x_{1:t}^{i},a_t^i)\\
	&= \sum_{a_{1:t-1}x_{1:t}^{i},a_t^i}\mP^{g^{-i}}(\tilde{x}_t,\tilde{a}_t|a_{1:t-1},x_t^i,a_t^i) \mP^g(a_{1:t-1},x_{1:t}^{i},a_t^i)\label{eq:D6}\\
	&= \sum_{a_{1:t-1}x^i_t ,a_t^i}\mP^{g^{-i}}(\tilde{x}_t,\tilde{a}_t|a_{1:t-1},x_t^i,a_t^i) \mP^g(a_{1:t-1},x^i_t,a_t^i) \label{eq:D7}\\
	&= \sum_{a_{1:t-1}x^i_t ,a_t^i}\mP^{g^{-i}}(\tilde{x}_t,\tilde{a}_t|a_{1:t-1},x_t^i,a_t^i) \mP^{s^ig^{-i}}(a_{1:t-1},x^i_t,a_t^i) \label{eq:D8}\\
	&= \mP^{s^ig^{-i}}(\tilde{x}_t, \tilde{a}_t).
}
}
where \eqref{eq:D6} follows from \eqref{eq:D5} in Claim~\ref{claim:B4} and \eqref{eq:D8} from \eqref{eq:D0} in Claim~\ref{claim:B3}.

\section{Proof of Lemma~\ref{fact:L1}}
\label{app:A2}
For this proof we will assume the common agents strategies to be probabilistic as opposed to being deterministic, as was the case in Section~\ref{sec:StructuralResults}. This means actions of the common agent, $\gamma_t^i$'s are generated probabilistically from $\psi^i$ as $\Gamma_t^i \sim \psi_t^i(\cdot|a_{1:t-1})$, as opposed to being deterministically generated as $\gamma_t^i = \psi_t^i[a_{1:t-1}]$, as before. These two are equivalent ways of generating actions $a_t^i$ from $a_{1:t-1}$ and $x_t^i$. We avoid using the probabilistic strategies of common agent throughout the main text for ease of exposition, and because it conceptually does not affect the results.

\begin{IEEEproof}
We prove this lemma in the following steps.
We view this problem from the perspective of a common agent. Let $\psi$ be the coordinator's policy corresponding to policy profile $g$. Let $\pi^{i}_t(x_t^i) = \mP^{\psi^i}(x_t^i|a_{1:t-1})$.

\begin{itemize}
\item[(a)] In Claim~\ref{claim:C1}, we show that $\pi_t$ can be factorized as $\pi_t(x_t) = \prod_{i=1}^N \pi_t^{{i}}(x_t^i)$ where each $\pi_t^{{i}}$ can be updated through an update function  $\pi^{i}_{t+1} = F^i(\pi_t^{{i}},\gamma^i_t,a_t)$ and $F^i$ is independent of common agent's policy $\psi$.
\item[(b)] In Claim~\ref{claim:C2}, we prove that $(\Pi_t, \Gamma_t )_{t \in \mathcal{T}}$ is a controlled Markov process.
\item[(c)] We construct a policy profile $\theta$ from $g$ such that  $ \theta_t(d\gamma_t|\pi_t) \defeq \mP^{\psi}(d\gamma_t|\pi_t)$.
\item[(d)] In Claim~\ref{claim:C3}, we prove that dynamics of this Markov process $(\Pi_t, \Gamma_t )_{t \in \mathcal{T}}$ under $\theta$ is same as under $\psi$ i.e. $\mP^{\theta}(d\pi_t ,d\gamma_t ,d\pi_{t+1} ) = \mP^{\psi} (d\pi_t, d\gamma_t, d\pi_{t+1})$.
\item[(e)] In Claim~\ref{claim:C4}, we prove that with respect to random variables $(X_t,A_t)$, $\pi_t$ can summarize common information $a_{1:t-1}$  i.e. $\mP^{\psi}(x_t, a_t|a_{1:t-1}, \gamma_t) = \mP(x_t,a_t|\pi_t,\gamma_t )$.
\item[(f)] From (c), (d) and (e) we then prove the result of the lemma that $\mP^{\psi}(x_t, a_t) = \mP^{\theta}(x_t, a_t)$ which is equivalent to $\mP^g(x_t, a_t) = \mP^m(x_t, a_t)$, where $m$ is the policy profile of players corresponding to $\theta$.\\
\end{itemize}

\begin{claim}
	$\pi_t$ can be factorized as $\pi_t(x_t) = \prod_{i=1}^N \pi_t^{{i}}(x_t^i)$ where each $\pi_t^{{i}}$ can be updated through an update function  $\pi^{i}_{t+1} = F^i(\pi^{i}_t,\gamma^i_t,a_t)$ and $F^i$ is independent of common agent's policy $\psi$.	
We also say $\underline{\pi}_{t+1} =  \underline{F}(\underline{\pi}_t,\gamma_t,a_t)$.

	\label{claim:C1}
	\end{claim}
	\begin{IEEEproof}
		We prove this by induction. Since $\pi_1(x_1) = \prod_{i=1}^N Q_t^i(x_1^i)$, the base case is verified. Now suppose  $\pi_{t} = \prod_{i=1}^N \pi_t^{{i}}$. Then,
		\seq{
		\label{eq:piupdate}
	\eq{
	\pi_{t+1}(x_{t+1}) &= \mP^{\psi}(x_{t+1}|a_{1:t},\gamma_{1:t+1}) \\
	&=  \mP^{\psi}(x_{t+1}|a_{1:t},\gamma_{1:t}) \\
	&= \frac{ \sum_{x_t} \mP^{\psi}(x_t,a_t,x_{t+1}|a_{1:t-1},\gamma_{1:t}) }%
	{ \sum_{\tilde{x}_{t+1} \tilde{x}_{t}}\mP^{\psi}(\tilde{x}_t, \tilde{x}_{t+1}, a_t |a_{1:t-1},\gamma_{1:t}) } \label{eq:C1b}\\
	&= \frac{   \sum_{x_t} \pi_t(x_t) \prod_{i=1}^N \gamma_t^i(a_t^i|x_t^i)Q_t^i(x_{t+1}^i|x_t^i, a_t) } %
	{ \sum_{\tilde{x}_t\tilde{x}_{t+1}}\pi_t(\tilde{x}_t) \prod_{i=1}^N \gamma_t^i(a_t^i|\tilde{x}_t^i) Q_t^i(\tilde{x}_{t+1}^i|\tilde{x}_t^i, a_t)}\label{eq:C1c} \\
	&= \prod_{i=1}^N  \frac{  \sum_{x^i_t} \pi_t^{{i}}(x_t^i) \gamma_t^i(a_t^i|x_t^i)Q_t^i(x_{t+1}^i|x_t^i, a_t) } %
	{ \sum_{\tilde{x}_t^i}\pi_t^{{i}}(\tilde{x}_t^i)  \gamma_t^i(a_t^i|\tilde{x}_t^i)}\label{eq:C1a}\\
	&= \prod_{i=1}^N \pi_{t+1}^{{i}}({x_{t+1}^i}), \label{eq:pi_prod}
	}
	}
	where \eqref{eq:C1a} follows from induction hypothesis.  It is assumed in \eqref{eq:C1b}-\eqref{eq:C1a} that the denominator is not 0. If denominator corresponding to any $\gamma_t^i$ is zero, we define
	\eq{
	\pi_{t+1}^i(x_{t+1}^i) &= \sum_{x_t^i} \pi_t^i(x_t^i) Q_t^i(x_{t+1}^i| x_t^i,a_t),
	}
	where $\pi_{t+1}$ still satisfies \eqref{eq:pi_prod}.
	Thus $\pi^i_{t+1} = F^i(\pi^{i}_t,\gamma^i_t,a_t)$ and $\underline{\pi}_{t+1}= \underline{F}(\underline{\pi}_t,\gamma_t,a_1)$ where $F^i$ and $\underline{F}$ are appropriately defined from above.
	\end{IEEEproof}

	\begin{claim}
	$(\Pi_t, \Gamma_t )_{t \in \mathcal{T}}$ is a controlled Markov process with state $\Pi_t$ and control action $\Gamma_t$
	\label{claim:C2}
	\end{claim}
	\begin{IEEEproof}
	\seq{
	\eq{
	\mP^{\psi}(d\pi_{t+1}|\pi_{1:t}, \gamma_{1:t}) &= \sum_{a_t,x_t} \mP^{\psi}(d\pi_{t+1},a_t,x_t|\pi_{1:t}, \gamma_{1:t}) \\
	&= \sum_{a_t,x_t} \mP^{\psi}(x_t| \pi_{1:t}, \gamma_{1:t}) \left\{\prod_{i=1}^N \gamma_t^i(a_t^i|x_t^i) \right\} I_{{F}(\pi_t, \gamma_t,a_t)}(\pi_{t+1} )\\
	&= \sum_{a_t,x_t} \pi_t(x_t)\left\{\prod_{i=1}^N \gamma_t^i(a_t^i|x_t^i) \right\} I_{{F}(\pi_t, \gamma_t,a_t)}(\pi_{t+1})\\
	&= \mP(d\pi_{t+1}|\pi_t, \gamma_t) \label{eq:piupdateeq}.
	}
	}
	\end{IEEEproof}

 For any given policy profile $\psi$, we construct policy profile $\theta$ in the following way.

	\eq{ \theta_t(d\gamma_t|\pi_t) &\defeq \mP^{\psi}(d\gamma_t|\pi_t) \label{eq:defMi}.
	}
%
	\begin{claim}
	\eq{
	\mP^{\psi}(d\pi_t ,d\gamma_t, d\pi_{t+1} ) &= \mP^{\theta} (d\pi_t, d\gamma_t ,d\pi_{t+1} ) \;\;\;\; \forall t \in \mathcal{T}. \label{eq:G0}
	}
	\label{claim:C3}
	\end{claim}
	\begin{IEEEproof}
	We prove this by induction. For $t=1$,
	\eq{
	\mP^{\psi}(d\pi_{1} ) &= \mP^{\theta}(d\pi_{1}) = I_Q(\pi_1). \label{eq:G1}
	}
	Now suppose $P^{\psi}(d\pi_{t}) = P^{{\theta}}(d\pi_{t})$ is true for $t$, then
		\seq{
	\eq{
	\mP^{\psi}(d\pi_t ,d\gamma_t ,d\pi_{t+1} ) &= \mP^{\psi}(d\pi_{t})P^{\psi}(d\gamma_t|  \pi_t )\mP^{\psi}(d\pi_{t+1}| \pi_t  \gamma_t )\\
	&= \mP^{{\theta}}(d\pi_{t}){\theta}_t(d\gamma_t|  \pi_t )P(d\pi_{t+1}| \pi_t , \gamma_t )\label{eq:G2}\\
	&= \mP^{{\theta}}(d\pi_t , d\gamma_t, d\pi_{t+1}).
	}
	}
	where  \eqref{eq:G2} is true from induction hypothesis, definition of ${\theta}$ in \eqref{eq:defMi} and since $(\Pi_t, \Gamma_t )_{t \in \mathcal{T}}$ is a controlled Markov process as proved in Claim~\ref{claim:C2} and thus its update kernel does not depend on policy ${\psi}$. This completes the induction step.
	\end{IEEEproof}
	\begin{claim}
	For any policy ${\psi}$,
	\eq{
	\mP^{\psi}(x_t, a_t|a_{1:t-1}, \gamma_t) = \mP(x_t,a_t|\pi_t,\gamma_t ). \label{eq:G5}
	}
	\label{claim:C4}
	\end{claim}
	\begin{IEEEproof}
	\seq{
	\eq{
	\mP^{\psi}(x_t, a_t|a_{1:t-1},\gamma_t )
	 &=  \mP^{{\psi}}(x_t|a_{1:t-1}, \gamma_{t}) \prod_{i \in \cN} \gamma^i_t(a_t^i|x_t^i)\\
&=  \pi_t(x_t) \prod_{i\in \cN}   \gamma^i_t(a_t^i|x_t^i)\\
&= \mP(x_t,a_t|\pi_t, \gamma_t) .
	}
	}
	\end{IEEEproof}
	Finally,
	\seq{
	\eq{\mP^{{\psi}}(x_t, a_t) &= \sum_{a_{1:t-1},\gamma_t}\mP^{\psi}(x_t,a_t|a_{1:t-1},\gamma_t) \mP^{\psi}(a_{1:t-1},\gamma_t)\\
	&= \sum_{a_{1:t-1}\gamma_t}\mP(x_t,a_t|\pi_t, \gamma_t) \mP^{\psi}(a_{1:t-1},\gamma_t)\label{eq:G6}\\
	&= \sum_{\pi_t, \gamma_t}\mP(x_t,a_t|\pi_t, \gamma_t) \mP^{\psi}(\pi_t, \gamma_t) \label{eq:G7}\\
	&= \sum_{\pi_t, \gamma_t}\mP(x_t,a_t|\pi_t, \gamma_t) \mP^{{\theta}}(\pi_t, \gamma_t) \label{eq:G8}\\
	&= \mP^{\theta}(x_t, a_t).
}
}
	where \eqref{eq:G6} follows from \eqref{eq:G5}, \eqref{eq:G7} is due to change of measure and \eqref{eq:G8} follows from \eqref{eq:G0}.
		
\end{IEEEproof}


\section{Proof of Theorem~\ref{Thm:Main}}
\label{app:B}

\begin{IEEEproof}
We prove \eqref{eq:prop} using induction and the results in Lemma~\ref{lemma:2}, \ref{lemma:3} and \ref{lemma:1} proved in Appendix~\ref{app:lemmas}.
\seq{
For base case at $t=T$, $\forall i\in \cN, (a_{1:T-1}, x_{1:T}^i)\in \mathcal{H}_{T}^i, \beta^i$
\eq{
\E^{\beta_{T}^{*,i} \beta_{T}^{*,-i},\, \mu_{T}^{*}[a_{1:T-1}] }\left\{  R_T^i(X_T,A_T) \big\lvert a_{1:T-1}, x_{1:T}^i \right\}
&=
V^i_T(\underline{\mu}_T^*[a_{1:T-1}], x_T^i)  \label{eq:T2a}\\
&\geq \E^{\beta_{T}^{i} \beta_{T}^{*,-i},\, \mu_{T}^{*}[a_{1:T-1}]} \left\{ R_T^i(X_T,A_T) \big\lvert a_{1:T-1}, x_{1:T}^i \right\},  \label{eq:T2}
}
}
where \eqref{eq:T2a} follows from Lemma~\ref{lemma:1} and \eqref{eq:T2} follows from Lemma~\ref{lemma:2} in Appendix~\ref{app:lemmas}.

Let the induction hypothesis be that for $t+1$, $\forall i\in \cN, a_{1:t} \in \mathcal{H}_{t+1}^c, x_{1:t+1}^i \in (\cX^i)^{t+1}, \beta^i$,
\seq{
\eq{
 \E^{\beta_{t+1:T}^{*,i} \beta_{t+1:T}^{*,-i},\, \mu_{t+1}^{*}[ a_{1:t}]} \left\{ \sum_{n=t+1}^T R_n^i(X_n,A_n) \big\lvert a_{1:t}, x_{1:t+1}^i \right\} \\
 \geq
  \E^{\beta_{t+1:T}^{i} \beta_{t+1:T}^{*,-i},\, \mu_{t+1}^{*}[ a_{1:t}]} \left\{ \sum_{n=t+1}^T R_n^i(X_n,A_n) \big\lvert  a_{1:t}, x_{1:t+1}^i \right\}. \label{eq:PropIndHyp}
}
}
\seq{
Then $\forall i\in \cN, (a_{1:t-1}, x_{1:t}^i) \in \mathcal{H}_{t}^i, \beta^i$, we have
\eq{
&\E^{\beta_{t:T}^{*,i} \beta_{t:T}^{*,-i},\, \mu_{t}^{*}[ a_{1:t-1}]} \left\{ \sum_{n=t}^T R_n^i(X_n,A_n) \big\lvert a_{1:t-1}, x_{1:t}^i \right\} \nonumber \\
&= V^i_t(\underline{\mu}^*_t[a_{1:t-1}], x_t^i)\label{eq:T1}\\
&\geq \E^{\beta_t^i \beta_t^{*,-i}, \,\mu_t^*[a_{1:t-1}]} \left\{ R_t^i(X_t,A_t) + V_{t+1}^i (\underline{\mu}^*_{t+1}[a_{1:t-1}A_t], X_{t+1}^i) \big\lvert a_{1:t-1}, x_{1:t}^i \right\}  \label{eq:T3}\\
&= \E^{\beta_t^i \beta_t^{*,-i}, \,\mu_t^*[a_{1:t-1}]} \left\{ R_t^i(X_t,A_t) + \E^{\beta_{t+1:T}^{*,i} \beta_{t+1:T}^{*,-i},\, \mu_{t+1}^{*}[ a_{1:t-1},A_t]} \left\{ \sum_{n=t+1}^T R_n^i(X_n,A_n) \big\lvert \rp\rp\nn\\
&\hs{2cm}\lp\lp a_{1:t-1},A_t, x_{1:t+1}^i \right\}  \big\vert a_{1:t-1}, x_{1:t}^i \right\}  \label{eq:T3b}\\
&\geq \E^{\beta_t^i \beta_t^{*,-i}, \,\mu_t^*[a_{1:t-1}]} \left\{ R_t^i(X_t,A_t) + \right.\nonumber \\
&\hspace{0cm} \left.\E^{\beta_{t+1:T}^{i} \beta_{t+1:T}^{*,-i} \mu_{t+1}^{*}[a_{1:t-1},A_t ]} \left\{ \sum_{n=t+1}^T R_n^i(X_n,A_n) \big\lvert a_{1:t-1},A_t, x_{1:t}^i,X_{t+1}^i\right\} \big\vert a_{1:t-1}, x_{1:t}^i \right\}  \label{eq:T4} \\
&= \E^{\beta_t^i \beta_t^{*,-i}, \,  \mu_t^*[a_{1:t-1}]} \big\{ R_t^i(X_t,A_t) +  
\nn \\
&\E^{\beta_{t:T}^{i} \beta_{t:T}^{*,-i} \mu_{t}^{*}[a_{1:t-1}]} 
\left\{ \sum_{n=t+1}^T R_n^i(X_n,A_n) \big\lvert a_{1:t-1},A_t, x_{1:t}^i,X_{t+1}^i\right\} \big\vert a_{1:t-1}, x_{1:t}^i \big\}  
\label{eq:T5}\\
&=\E^{\beta_{t:T}^{i} \beta_{t:T}^{*,-i}\, \mu_{t}^{*}[a_{1:t-1}]} \left\{ \sum_{n=t}^T R_n^i(X_n,A_n) \big\lvert a_{1:t-1},  x_{1:t}^i \right\}  \label{eq:T6},
}
}
where \eqref{eq:T1} follows from Lemma~\ref{lemma:1}, \eqref{eq:T3} follows from Lemma~\ref{lemma:2}, \eqref{eq:T3b} follows from Lemma~\ref{lemma:1}, \eqref{eq:T4} follows from induction hypothesis in \eqref{eq:PropIndHyp} and \eqref{eq:T5} follows from Lemma~\ref{lemma:3}. Moreover, construction of $\theta$ in \eqref{eq:m_FP}, and consequently definition of $\beta^*$ in \eqref{eq:beta*def} are pivotal for \eqref{eq:T5} to follow from \eqref{eq:T4}.

We note that $\mu^*$ satisfies the consistency condition of~\cite[p. 331]{FuTi91book} from the fact that (a) for all $t$ and for every common history $a_{1:t-1}$, all players use the same belief $\mu_t^*[a_{1:t-1}]$ on $x_t$ and (b) the belief $\mu_t^*$ can be factorized as $\mu_t^*[a_{1:t-1}] = \prod_{i=1}^N \mu_t^{*,{i}}[a_{1:t-1}] \; \forall a_{1:t-1} \in \mathcal{H}_t^c$ where $\mu_t^{*,{i}}$ is updated through Bayes' rule $F^i$ as in Lemma~\ref{claim:C1} in Appendix~\ref{app:A2}.
\end{IEEEproof}

\section{Intermediate lemmas used in proof of Theorem~\ref{Thm:Main}}
\label{app:lemmas}
\begin{lemma}
\label{lemma:2}
$\forall t\in \mathcal{T}, i\in \cN, (a_{1:t-1}, x_{1:t}^i)\in \mathcal{H}_t^i, \beta^i_t$
\begin{multline}
V_t^i(\underline{\mu}_t^*[a_{1:t-1}], x_t^i) \geq
\\ 
\E^{\beta_t^i \beta_t^{*,-i},\, \mu_t^*[a_{1:t-1}]} \left\{ R_t^i(X_t,A_t) + V_{t+1}^i (\uF(\underline{\mu}_t^*[a_{1:t-1}], \beta_t^*(\cdot|a_{1:t-1},\cdot), A_t), X_{t+1}^i) \big\lvert  a_{1:t-1}, x_{1:t}^i \right\}.\label{eq:lemma2}
\end{multline} 
\end{lemma}

\begin{IEEEproof}
We prove this Lemma by contradiction.

 Suppose the claim is not true for $t$. This implies $\exists i, \hat{\beta}_t^i, \hat{a}_{1:t-1}, \hat{x}_{1:t}^i$ such that
\begin{multline}
\E^{\hat{\beta}_t^i \beta_t^{*,-i},\, \mu_t^*[\hat{a}_{1:t-1}] } \left\{ R_t^i(X_t,A_t) +  V_{t+1}^i (\uF(\underline{\mu}_t^*[\hat{a}_{1:t-1}], \beta_t^*(\cdot|\hat{a}_{1:t-1},\cdot), A_t), X_{t+1}^i) \big\lvert \hat{a}_{1:t-1},\hat{x}_{1:t}^i \right\} \\
> V_t^i(\underline{\mu}_t^*[\hat{a}_{1:t-1}], \hat{x}_{t}^i).\label{eq:E8}
\end{multline}
We will show that this leads to a contradiction.

Construct 
\begin{equation}
\hat{\gamma}^i_t(a_t^i|x_t^i) = \lb{\hat{\beta}_t^i(a_t^i|\hat{a}_{1:t-1},\hat{x}_{1:t}^i) \;\;\;\;\; x_t^i = \hat{x}_t^i \\ \text{arbitrary} \;\;\;\;\;\;\;\;\;\;\;\;\;\; \text{otherwise.}  }
\end{equation}

Then for $\hat{a}_{1:t-1}, \hat{x}_{1:t}^i$, we have
\seq{
\eq{
&V_t^i(\underline{\mu}_t^*[\hat{a}_{1:t-1}], \hat{x}_t^i) 
\nn \\
&= \max_{\gamma^i_t(\cdot|\hat{x}_t^i)} \E^{\gamma^i_t(\cdot|\hat{x}_t^i) \beta_t^{*,-i}, \, \mu_t^*[\hat{a}_{1:t-1}]} \left\{ R_t^i(\hat{x}_t^ix_t^{-i},a_t) + V_{t+1}^i (\uF(\underline{\mu}_t^*[\hat{a}_{1:t-1}], \beta_t^{*}(\cdot|\hat{a}_{1:t-1},\cdot), A_t), X_{t+1}^i) \big\lvert  \hat{x}_{t}^i \right\}, \label{eq:E11}\\
&\geq\E^{\hat{\gamma}_t^i(\cdot|\hat{x}_t^i) \beta_t^{*,-i},\,\mu_t^*[\hat{a}_{1:t-1}]} \left\{ R_t^i(X_t,A_t) + V_{t+1}^i (\uF(\underline{\mu}_t^*[\hat{a}_{1:t-1}], \beta_t^{*}(\cdot|\hat{a}_{1:t-1},\cdot), A_t), {X}_{t+1}^i) \big\lvert \hat{x}_{t}^i \right\}   
\\ \nn 
&=\sum_{x_t^{-i},a_t,x_{t+1}}   \left\{ R_t^i(\hat{x}_t^ix_t^{-i},a_t) + V_{t+1}^i (\uF(\underline{\mu}_t^*[\hat{a}_{1:t-1}], \beta_t^{*}(\cdot|\hat{a}_{1:t-1},\cdot), a_t), x_{t+1}^i)\right\}
\\
&\qquad \times \mu_t^{*,-i}[\hat{a}_{1:t-1}] (x_t^{-i}) \hat{\gamma}_t^i(a^i_t|\hat{x}_t^i) \beta_t^{*,-i}(a_t^{-i}|\hat{a}_{1:t-1}, x_t^{-i})Q_t^i(x_{t+1}^i|\hat{x}_t^i,a_t)  
\\ \nn 
&= \sum_{x_t^{-i},a_t,x_{t+1}}  \left\{ R_t^i(\hat{x}_t^ix_t^{-i},a_t) + V_{t+1}^i (\uF(\underline{\mu}_t^*[\hat{a}_{1:t-1}], \beta_t^{*}(\cdot|\hat{a}_{1:t-1},\cdot), a_t), x_{t+1}^i)\right\}
\\
& \qquad \times  \mu_t^{*,-i}[\hat{a}_{1:t-1}](x_t^{-i}) \hat{\beta}^i_t(a_t^i|\hat{a}_{1:t-1} ,\hat{x}_{1:t}^i) \beta_t^{*,-i}(a_t^{-i}|\hat{a}_{1:t-1}, x_t^{-i}) Q_t^i(x_{t+1}^i|\hat{x}_t^i,a_t) \label{eq:E9}\\
&= \E^{\hat{\beta}_t^i \beta_t^{*,-i}, \mu_t^*[\hat{a}_{1:t-1}]} \left\{ R_t^i(\hat{x}_t^ix_t^{-i},a_t)+ V_{t+1}^i (\uF(\underline{\mu}_t^*[\hat{a}_{1:t-1}], \beta_t^{*}(\cdot|\hat{a}_{1:t-1},\cdot), A_t), X_{t+1}^i) \big\lvert \hat{a}_{1:t-1},  \hat{x}_{1:t}^i \right\}  \\
&> V_t^i(\underline{\mu}_t^*[\hat{a}_{1:t-1}], \hat{x}_{t}^i), \label{eq:E10}
}
where \eqref{eq:E11} follows from definition of $V_t^i$ in \eqref{eq:Vdef}, \eqref{eq:E9} follows from definition of $\hat{\gamma}_t^i$ and \eqref{eq:E10} follows from \eqref{eq:E8}. However this leads to a contradiction.
}
\end{IEEEproof}

\begin{lemma}
\label{lemma:3}
$\forall i\in \cN, t\in \mathcal{T}, (a_{1:t}, x_{1:t+1}^i)\in \mathcal{H}_{t+1}^i$ and
$\beta^i_{t}$
\seq{ 
\eq{
\E^{ \beta_{t:T}^{i}  \beta^{*,-i}_{t:T},\,\mu_{t}^{*}[a_{1:t-1}]}  \left\{ \sum_{n=t+1}^T R_n^i(X_n,A_n) \big\lvert  a_{1:t}, x_{1:t+1}^i \right\} \\ 
=
\E^{\beta^i_{t+1:T} \beta^{*,-i}_{t+1:T},\, \mu_{t+1}^{*}[a_{1:t}]}  \left\{ \sum_{n=t+1}^T R_n^i(X_n,A_n) \big\lvert a_{1:t}, x_{1:t+1}^i \right\}. \label{eq:F1}
}
}
Thus the above quantities do not depend on $\beta_t^i$.
\end{lemma}
\begin{IEEEproof} 
Essentially this claim stands on the fact that $\mu_{t+1}^{*,-i}[a_{1:t}]$ can be updated from $\mu_t^{*,-i}[a_{1:t-1}], \beta_t^{*,-i}$ and $a_t$, as $\mu_{t+1}^{*,-i}[a_{1:t}] = \prod_{j\neq i} F(\mu_t^{*,j}[a_{1:t-1}], \beta_t^{*,j}, a_t)$ as in Claim~\ref{claim:C1}.
Since the above expectations involve random variables $X_{t+1}^{-i}, A_{t+1:T}, X_{t+2:T}$, we consider the probability 
%
\begin{equation} 
\mP^{\beta^i_{t:T} \beta^{*,-i}_{t:T},\, \mu_{t}^{*}[a_{1:t-1}]} (x_{t+1}^{-i}, a_{t+1:T} ,x_{t+2:T}\lvert  a_{1:t}, x_{1:t+1}^i ) = \frac{Nr}{Dr} \label{eq:F2}
\end{equation}
where
\seq{
\eq{
&Nr =\sum_{x_t^{-i}}\mP^{\beta^i_{t:T} \beta^{*,-i}_{t:T},\, \mu_{t}^{*}[a_{1:t-1}]} (x_t^{-i} ,a_t, x_{t+1}, a_{t+1:T} ,x_{t+2:T} \vert a_{1:t-1}, x_{1:t}^i ) \\
&= \sum_{x_t^{-i}}\mP^{\beta^i_{t:T} \beta^{*,-i}_{t:T},\, \mu_{t}^{*}[a_{1:t-1}]} (x_t^{-i} \big\lvert  a_{1:t-1}, x_{1:t}^i ) 
\nonumber 
\\
&\beta_t^{i}(a_t^{i}|a_{1:t-1}, x_{1:t}^{i})\beta_t^{*,-i}(a_t^{-i}|a_{1:t-1}, x_t^{-i}) Q(x_{t+1}|x_t, a_t)\mP^{\beta^i_{t:T} \beta^{*,-i}_{t:T},\, \mu_{t}^{*}[a_{1:t-1}]} (a_{t+1:T}, x_{t+2:T}| a_{1:t} ,x_{1:t-1}^i, x_{t:t+1}) 
\\
=&\sum_{x_t^{-i}}\mu_t^{*,-i}[a_{1:t-1}](x_t^{-i})\beta_t^{i}(a_t^{i}|a_{1:t-1}, x_{1:t}^{i}) 
\nonumber 
\\  
& \beta_t^{*,-i}(a_t^{-i}|a_{1:t-1}, x_t^{-i}) Q^i(x^i_{t+1}|x^i_t, a_t)Q^{-i}(x^{-i}_{t+1}|x^{-i}_t, a_t)\mP^{\beta^i_{t+1:T} \beta^{*,-i}_{t+1:T},\, \mu_{t+1}^{*}[a_{1:t}]} (a_{t+1:T}, x_{t+2:T}| a_{1:t} ,x_{1:t}^i, x_{t+1}),\label{eq:Nr2}
}
where \eqref{eq:Nr2} follows from the conditional independence of types given common information, as shown in Claim~\ref{claim:CondInd}, 
and the fact that probability on $(a_{t+1:T} ,x_{2+t:T})$ given $a_{1:t}, x_{1:t-1}^i,x_{t:t+1}, \mu_{t}^{*}[a_{1:t-1}] $ depends on $a_{1:t}, x_{1:t}^i,x_{t+1}, \mu_{t+1}^{*}[a_{1:t}] $ through ${\beta_{t+1:T}^{ i} \beta_{t+1:T}^{*,-i} }$. Similarly, the denominator in \eqref{eq:F2} is given by
\eq{
Dr &= \sum_{\tilde{x}_{t}^{-i}} \mP^{\beta^i_{t:T} \beta^{*,-i}_{t:T},\, \mu_{t}^{*}[a_{1:t-1}]} (\tilde{x}_t^{-i}, a_t, x_{t+1}^i\big\lvert  a_{1:t-1}, x_{1:t}^i )\nn\\
&\sum_{\tilde{x}_{t}^{-i}} \mP^{\beta^i_{t:T} \beta^{*,-i}_{t:T} ,\,\mu_{t}^{*}} (\tilde{x}_t^{-i} | a_{1:t-1}, x_{1:t}^i )  \beta_t^{i}(a_t^{i}|a_{1:t-1}, x_{1:t}^{i}) \beta_t^{*,-i}(a_t^{-i}|a_{1:t-1}, \tilde{x}_t^{-i})Q^i(x^i_{t+1}|x^i_t, a_t)\label{eq:F3}\\
=&\sum_{\tilde{x}_{t}^{-i}} \mu_t^{*,-i}[a_{1:t-1}](\tilde{x}_t^{-i}) \beta_t^{i}(a_t^{i}|a_{1:t-1}, x_{1:t}^{i}) \beta_t^{*,-i}(a_t^{-i}|a_{1:t-1}, \tilde{x}_t^{-i})Q^i(x^i_{t+1}|x^i_t, a_t).  \label{eq:F4}
%
}

By canceling the terms $\beta_t^i(\cdot)$ and $Q^i(\cdot)$ in the numerator and the denominator, \eqref{eq:F2} is given by
\eq{
&\frac{\sum_{x_t^{-i}}\mu_t^{*,-i}[a_{1:t-1}](x_t^{-i}) \beta_t^{*,-i}(a_t^{-i}|a_{1:t-1}, x_t^{-i}) Q_{t+1}^{-i}(x^{-i}_{t+1}|x^{-i}_t, a_t)}{\sum_{\tilde{x}_{t}^{-i}} \mu_t^{*,-i}[a_{1:t-1}](\tilde{x}_t^{-i}) \beta_t^{*,-i}(a_t^{-i}|a_{1:t-1}, \tilde{x}_t^{-i})}  
\nonumber \\
&\times\mP^{\beta^i_{t+1:T} \beta^{*,-i}_{t+1:T},\, \mu_{t+1}^{*}[a_{1:t}]} (a_{t+1:T}, x_{t+2:T}| a_{1:t} ,x_{1:t}^i, x_{t+1})
\\
&=\mu_{t+1}^{*,-i}[a_{1:t}](x_{t+1}^{-i})
\mP^{\beta^i_{t+1:T} \beta^{*,-i}_{t+1:T},\, \mu_{t+1}^{*}[a_{1:t}]} (a_{t+1:T}, x_{t+2:T}| a_{1:t} ,x_{1:t}^i, x_{t+1})
\label{eq:F6}
\\
&= \mP^{\beta_{t+1:T}^{ i} \beta_{t+1:T}^{*, -i},\, \mu_{t+1}^{*}[a_{1:t}]}  (x_{t+1}^{-i} ,a_{t+1:T},x_{t+2:T} | a_{1:t}, x_{1:t+1}^i ),
}
}
where \eqref{eq:F6} follows from using the definition of $\mu_{t+1}^{*,-i}[a_{1:t}](x_t^{-i})$ in the forward recursive step in \eqref{eq:mu*def} and the definition of the belief update in \eqref{eq:piupdate}.
\end{IEEEproof}

\begin{lemma}
\label{lemma:1}
$\forall i\in \cN, t\in \mathcal{T}, (a_{1:t-1}, x_{1:t}^i)\in \mathcal{H}_t^i$,
\begin{gather}
V^i_t(\underline{\mu}^*_t[a_{1:t-1}], x_t^i) =
\E^{\beta_{t:T}^{*,i} \beta_{t:T}^{*,-i},\mu_{t}^{*}[a_{1:t-1}]} \left\{ \sum_{n=t}^T R_n^i(X_n,A_n) \big\lvert  a_{1:t-1}, x_{1:t}^i \right\} .
\end{gather} 
\end{lemma}

\begin{IEEEproof}
%
\seq{
We prove the lemma by induction. For $t=T$,
\eq{
 &\E^{\beta_{T}^{*,i} \beta_{T}^{*,-i} ,\,\mu_{T}^{*}[a_{1:T-1}]} \left\{  R_T^i(X_T,A_T) \big\lvert a_{1:T-1},  x_{1:T}^i \right\}
 \nonumber \\ 
 &= \sum_{x_T^{-i} a_T} R_T^i(x_T,a_T)\mu_{T}^{*}[a_{1:T-1}](x_T^{-i}) \beta_{T}^{*,i}(a_T^i|a_{1:T-1},x_{T}^i) 
 \beta_{T}^{*,-i}(a_T^{-i}|a_{1:T-1}, x_{T}^{-i})\\
 &= V^i_T(\underline{\mu}^*_T[a_{1:T-1}], x_T^i) \label{eq:C1},
}
}
where \eqref{eq:C1} follows from the definition of $V_t^i$ in \eqref{eq:Vdef} and the definition of $\beta_T^*$ in the forward recursion in \eqref{eq:beta*def}.

Suppose the claim is true for $t+1$, i.e., $\forall i\in \cN, t\in \mathcal{T}, (a_{1:t}, x_{1:t+1}^i)\in \mathcal{H}_{t+1}^i$
\begin{gather}
V^i_{t+1}(\underline{\mu}^*_{t+1}[a_{1:t}], x_{t+1}^i) = \E^{\beta_{t+1:T}^{*,i} \beta_{t+1:T}^{*,-i},\, \mu_{t+1}^{*}[a_{1:t}]} 
\left\{ \sum_{n=t+1}^T R_n^i(X_n,A_n) \big\lvert a_{1:t}, x_{1:t+1}^i \right\} 
\label{eq:CIndHyp}.
\end{gather}
Then $\forall i\in \cN, t\in \mathcal{T}, (a_{1:t-1}, x_{1:t}^i)\in \mathcal{H}_t^i$, we have
\seq{
\eq{
&\E^{\beta_{t:T}^{*,i} \beta_{t:T}^{*,-i} ,\,\mu_{t}^{*}[a_{1:t-1}]} \left\{ \sum_{n=t}^T R_n^i(X_n,A_n) \big\lvert  a_{1:t-1}, x_{1:t}^i \right\} 
\nonumber 
\\
&=  \E^{\beta_{t:T}^{*,i} \beta_{t:T}^{*,-i} ,\,\mu_{t}^{*}[a_{1:t-1}]} \left\{R_t^i(X_t,A_t) +\E^{\beta_{t:T}^{*,i} \beta_{t:T}^{*,-i} ,\,\mu_{t}^{*}[a_{1:t-1}]}  \right.
\nonumber \\ 
&\left. \left\{ \sum_{n=t+1}^T R_n^i(X_n,A_n)\big\lvert a_{1:t-1},  A_t, x_{1:t}^i,X_{t+1}^i\right\} \big\lvert a_{1:t-1},  x_{1:t}^i \right\} \label{eq:C2}
\\
&=  \E^{\beta_{t:T}^{*,i} \beta_{t:T}^{*,-i} ,\,\mu_{t}^{*}[a_{1:t-1}]} \left\{R_t^i(X_t,A_t) +\right.
\nonumber 
\\
&\hspace{0cm}\left.  \left\{ \sum_{n=t+1}^T R_n^i(X_n,A_n)\big\lvert a_{1:t-1},A_t, x_{1:t}^i,X_{t+1}^i\right\} \big\lvert a_{1:t-1}, x_{1:t}^i \right\} \label{eq:C3}
\\
&=  \E^{\beta_{t:T}^{*,i} \beta_{t:T}^{*,-i} ,\,\mu_{t}^{*}[a_{1:t-1}]} \left\{R_t^i(X_t,A_t) +  V^i_{t+1}(\underline{\mu}^*_{t+1}[a_{1:t-1}A_t], X_{t+1}^i) \big\lvert  a_{1:t-1}, x_{1:t}^i \right\} 
\label{eq:C4}
\\
&=  \E^{\beta_{t}^{*,i} \beta_{t}^{*,-i} ,\,\mu_{t}^{*}[a_{1:t-1}]} \left\{R_t^i(X_t,A_t) +  V^i_{t+1}(\underline{\mu}^*_{t+1}[a_{1:t-1}A_t], X_{t+1}^i) \big\lvert  a_{1:t-1}, x_{1:t}^i \right\} 
\label{eq:C5}
\\
&=V^i_{t}(\underline{\mu}^*_{t}[a_{1:t-1}], x_t^i) \label{eq:C6},
}
}
where \eqref{eq:C3} follows from Lemma~\ref{lemma:3} in Appendix~\ref{app:lemmas}, \eqref{eq:C4} follows from the induction hypothesis in \eqref{eq:CIndHyp}, \eqref{eq:C5} follows because the random variables involved in expectation, $X_t^{-i},A_t,X_{t+1}^i$ do not depend on $\beta_{t+1:T}^{*,i} \beta_{t+1:T}^{*,-i}$ and \eqref{eq:C6} follows from the definition of $\beta_t^*$ in the forward recursion in~\eqref{eq:beta*def}, the definition of $\mu_{t+1}^*$ in \eqref{eq:mu*def} and the definition of $V_t^i$ in \eqref{eq:Vdef}.
\end{IEEEproof}

\section{Proof of Theorem~\ref{thm:2}} \label{app:id}

\begin{IEEEproof}
We prove this by contradiction. Suppose for any equilibrium generating function $\phi$ that generates $(\beta^*,\mu^*)$ through forward recursion, there exists $t\in\cT, i\in\cN, a_{1:t-1}\in\cH_t^c,$ such that for $ \underline{\pi}_t =\underline{\mu}^*_t[a_{1:t-1}] $, \eqref{eq:m_FP} is not satisfied for $\phi$
i.e. for $\tgamma_t^i = \phi^i[\underline{\pi}_t] = \beta_t^{*,i}(\cdot|\underline{\mu}^*_t[a_{1:t-1}],x_t^i)$,
\eq{
 \tilde{\gamma}^{i}_t \not\in \arg\max_{\gamma^i_t} \E^{\gamma^i_t(\cdot|x^i) \tilde{\gamma}^{-i}_t, \, \pi_t} \left\{ R_t^i(X_t,A_t) 
+ V_{t+1}^i (\underline{F}(\underline{\pi}_t,\tilde{\gamma}_t, A_t), X_{t+1}^i) \big\lvert  x_t^i \right\}. \label{eq:FP4}
  }
  Let $t$ be the first instance in the backward recursion when this happens. This implies $\exists\ \hat{\gamma}_t^i$ such that
  \eq{
  \E^{\hat{\gamma}^i_t(\cdot|x^i) \tilde{\gamma}^{-i}_t, \, \pi_t} \left\{ R_t^i(X_t,A_t)+ V_{t+1}^i (\underline{F}(\underline{\pi}_t, \tilde{\gamma}_t, A_t), X_{t+1}^i) \big\lvert  x_t^i \right\}
  \nn\\
  > \E^{\tgamma^i_t(\cdot|x^i) \tilde{\gamma}^{-i}_t, \, \pi_t} \left\{ R_t^i(X_t,A_t) + V_{t+1}^i (\underline{F}(\underline{\pi}_t, \tilde{\gamma}_t, A_t), X_{t+1}^i) \big\lvert  x_t^i \right\} \label{eq:E1}
  }
  This implies for $\hat{\beta}_t(\cdot|\underline{\mu}^*_t[a_{1:t-1}],\cdot) = \hat{\gamma}_t^i$,
  \seq{
  \eq{
  &\E^{\beta_{t:T}^{*,i} \beta_{t:T}^{*,-i},\,\mu_{t}^{*}[a_{1:t-1}]} \left\{ \sum_{n=t}^T R_n^i(X_n,A_n) \big\lvert  a_{1:t-1}, x_{1:t}^i \right\}
  \nn\\
  &= \E^{\beta_t^{*,i} \beta_t^{*,-i}, \,\mu_t^*[a_{1:t-1}]} \left\{ R_t^i(X_t,A_t) + \E^{\beta_{t:T}^{*,i} \beta_{t:T}^{*,-i},\, \mu_{t}^{*}[ a_{1:t-1}]} \right. \nonumber \\
&\left.  \left\{ \sum_{n=t+1}^T R_n^i(X_n,A_n) \big\lvert a_{1:t-1},A_t, x_{1:t+1}^i \right\}  \big\vert a_{1:t-1}, x_{1:t}^i \right\}
\\
  &= \E^{\beta_t^{*,i} \beta_t^{*,-i}, \,\mu_t^*[a_{1:t-1}]} \left\{ R_t^i(X_t,A_t) + \E^{\beta_{t+1:T}^{*,i} \beta_{t+1:T}^{*,-i},\, \mu_{t+1}^{*}[ a_{1:t-1},A_t]} \right. \nonumber \\
&\left.  \left\{ \sum_{n=t+1}^T R_n^i(X_n,A_n) \big\lvert a_{1:t-1},A_t, x_{1:t+1}^i \right\}  \big\vert a_{1:t-1}, x_{1:t}^i \right\} \label{eq:E2}
  \\
  &=\E^{\tgamma^i_t(\cdot|x_t^i) \tilde{\gamma}^{-i}_t, \, \pi_t} \left\{ R_t^i(X_t,A_t) + V_{t+1}^i (\underline{F}(\underline{\pi}_t, \tilde{\gamma}_t, A_t), X_{t+1}^i) \big\lvert  x_t^i \right\} \label{eq:E3}
  \\
  &< \E^{\hat{\beta}^i_t(\cdot|\underline{\mu}^*_t[a_{1:t-1}],x_t^i) \tilde{\gamma}^{-i}_t, \, \pi_t} \left\{ R_t^i(X_t,A_t) + V_{t+1}^i (\underline{F}(\underline{\pi}_t, \tilde{\gamma}_t, A_t), X_{t+1}^i) \big\lvert  x_t^i \right\}\label{eq:E4}
  \\
  &= \E^{\hat{\beta}_t^i \beta_t^{*,-i}, \,  \mu_t^*[a_{1:t-1}]} \left\{ R_t^i(X_t,A_t) +  \E^{\beta_{t+1:T}^{*,i} \beta_{t+1:T}^{*,-i} \mu_{t+1}^{*}[a_{1:t-1},A_t]}\right. \nn \\
&\qquad \left. \left\{ \sum_{n=t+1}^T R_n^i(X_n,A_n) \big\lvert a_{1:t-1},A_t, x_{1:t}^i,X_{t+1}^i\right\} \big\vert a_{1:t-1}, x_{1:t}^i \big\} \right.\label{eq:E5}
  \\
  &=\E^{\hat{\beta}_t^i,\beta_{t+1:T}^{*,i} \beta_{t:T}^{*,-i},\,\mu_{t}^{*}[a_{1:t-1}]} \left\{ \sum_{n=t}^T R_n^i(X_n,A_n) \big\lvert  a_{1:t-1}, x_{1:t}^i \right\},\label{eq:E6}
  }
}
  where \eqref{eq:E2} follows from Lemma~\ref{lemma:3}, \eqref{eq:E3} follows from the definitions of $\tgamma_t^i$ and $\mu^*_{t+1}[a_{1:t-1},A_t]$ and Lemma~\ref{lemma:1}, \eqref{eq:E4} follows from \eqref{eq:E1} and the definition of $\hat{\beta}_t^i$, \eqref{eq:E5} follows from Lemma~\ref{lemma:2}, \eqref{eq:E6} follows from Lemma~\ref{lemma:3}. However, this leads to a contradiction since $(\beta^*,\mu^*)$ is a PBE of the game.
\end{IEEEproof}

\section{Proof of Theorem~\ref{thih}} \label{app:ih}

	We divide the proof into two parts: first we show that the value function $ V^i $ is at least as big as any reward-to-go function; secondly we show that under the strategy $ \beta_i^\star $, reward-to-go is $ V^i $.	

\paragraph*{Part 1}
For any $ i \in \mN $, $ \beta^i $ define the following reward-to-go functions
\begin{subequations} \label{eqihr2g}
\begin{gather}
	W_t^{i,\beta^i}(h_t^i) = \mE^{\beta^i,\beta^{-i,\star},\mu_t^\star[h_t^c]} \lpr \sum_{n=t}^\infty \delta^{n-t} R^i(X_n,A_n) \mid h_t^i \rpr
\end{gather}\vspace{-4ex}
\begin{gather}
	W_t^{i,\beta^i,T}(h_t^i) = \mE^{\beta^i,\beta^{-i,\star},\mu_t^\star[h_t^c]} \lpr \sum_{n=t}^T \delta^{n-t} R^i(X_n,A_n)
	+  \delta^{T+1-t} V^{i}(\underline{\Pi}_{T+1},X^i_{T+1}) \mid h_t^i \rpr.
\end{gather}
\end{subequations}
Since $ \mX^i,\mA^i $ are finite sets the reward $ R^i $ is absolutely bounded, the reward-to-go $ W_t^{i,\beta^i}(h_t^i) $ is finite $ \forall $ $ i,t,\beta^i,h_t^i $.

For any $ i \in \mN $, $ h_t^i \in \mathcal{H}_t^i $,
\begin{multline}\label{eqdc}
V^i\big(\underline{\mu}_t^\star[h_t^c],x_t^i\big) - W_t^{i,\beta^i}(h_t^i)
= \Big[ V^i\big(\underline{\mu}_t^\star[h_t^c],x_t^i\big) - W_t^{i,\beta^i,T}(h_t^i) \Big]
+ \Big[ W_t^{i,\beta^i,T}(h_t^i) - W_t^{i,\beta^i}(h_t^i) \Big]
	\end{multline}
Combining results from Lemmas~\ref{thmfh2} and~\ref{lemfhtoih} in Appendix~\ref{app:interm_ih}, %
the term in the first bracket in RHS of~\eqref{eqdc} is non-negative. Using~\eqref{eqihr2g}, the term in the second bracket is
\begin{gather}\label{eqdiff}	
\left( \delta^{T+1-t} \right) \mE^{\beta^i,\beta^{-i,\star},\mu_t^\star[h_t^c]} \Big\{- \sum_{n=T+1}^\infty \delta^{n-(T+1)} R^i(X_n,A_n)
+ V^{i}(\underline{\Pi}_{T+1},X^i_{T+1}) \mid h_t^i \Big\}.
\end{gather} 	
The summation in the expression above is bounded by a convergent geometric series. Also, $ V^i $ is bounded. Hence the above quantity can be made arbitrarily small by choosing $ T $ appropriately large. Since the LHS of~\eqref{eqdc} does not depend on $ T $, this results in
\begin{gather}
V^i\big(\underline{\mu}_t^\star[h_t^c],x_t^i\big) \ge W_t^{i,\beta^i}(h_t^i).
\end{gather}

\paragraph*{Part 2}
Since the strategy $ \beta^\star $ generated in~\eqref{eqbeta} is such that $\beta^{i,\star}_t $ depends on $ h_t^i $ only through $ \underline{\mu}_t^\star[h_t^c] $ and $ x_t^i $, the reward-to-go $ W_t^{i,\beta^{i,\star}} $, at strategy $ \beta^\star $, can be written (with abuse of notation) as
\begin{gather}
W_t^{i,\beta^{i,\star}}(h_t^i) = W_t^{i,\beta^{i,\star}}(\underline{\mu}_t^\star[h_t^c],x_t^i)
= \mE^{\beta^{\star},\mu_t^\star[h_t^c]} \lpr \sum_{n=t}^\infty \delta^{n-t} R^i(X_n,A_n) \mid \underline{\mu}_t^\star[h_t^c],x_t^i \rpr.
\end{gather}

For any $ h_t^i \in \mathcal{H}_t^i $,
\begin{subequations}
	\begin{multline}
	W_t^{i,\beta^{i,\star}}(\underline{\mu}_t^\star[h_t^c],x_t^i)
	= \mE^{\beta^{\star},\mu_t^\star[h_t^c]} \lpr R^i(X_t,A_t)	+ \delta W_{t+1}^{i,\beta^{i,\star}}
	\big(\underline{F}(\underline{\mu}_t^{\star}[h_t^c],\theta[\underline{\mu}_t^\star[h_t^c]],A_{t+1}),X_{t+1}^i\big)  \mid \underline{\mu}_t^\star[h_t^c],x_t^i \rpr
	\end{multline}\vspace{-3ex}
	\begin{multline}
	V^{i}(\underline{\mu}_t^\star[h_t^c],x_t^i)
	= \mE^{\beta^{\star},\mu_t^\star[h_t^c]} \Big\{ R^i(X_t,A_t) + \delta V^{i}
	\big(\underline{F}(\underline{\mu}_t^{\star}[h_t^c],\theta[\underline{\mu}_t^\star[h_t^c]],A_{t+1}),X_{t+1}^i\big)  \mid \underline{\mu}_t^\star[h_t^c],x_t^i \Big\}.
	\end{multline}
\end{subequations}
Repeated application of the above for the first $ n $ time periods gives
\begin{subequations}
	\begin{multline}
	W_t^{i,\beta^{i,\star}}(\underline{\mu}_t^\star[h_t^c],x_t^i)
	= \mE^{\beta^{\star},\mu_t^\star[h_t^c]} \Bigg\{ \sum_{m=t}^{t+n-1} \delta^{m-t} R^i(X_t,A_t)
	+ \delta^{n}  W_{t+n}^{i,\beta^{i,\star}}\big(\underline{\Pi}_{t+n},X_{t+n}^i\big)  \mid \underline{\mu}_t^\star[h_t^c],x_t^i \Bigg\}
	\end{multline}\vspace{-7ex}
	\begin{multline}
\\
	V^{i}(\underline{\mu}_t^\star[h_t^c],x_t^i)
	= \mE^{\beta^{\star},\mu_t^\star[h_t^c]} \Bigg\{ \sum_{m=t}^{t+n-1} \delta^{m-t} R^i(X_t,A_t)
	+ \delta^{n}  V^{i}\big(\underline{\Pi}_{t+n},X_{t+n}^i\big)  \mid \underline{\mu}_t^\star[h_t^c],x_t^i \Bigg\}.
	\end{multline}
\end{subequations}
Here $ \underline{\Pi}_{t+n} $ is the $ n-$step belief update under strategy and belief prescribed by $\beta^\star,\mu^\star$.

Taking differences results in
	\begin{multline}
	W_t^{i,\beta^{i,\star}}(\underline{\mu}_t^\star[h_t^c],x_t^i)  - V^{i}(\underline{\mu}_t^\star[h_t^c],x_t^i)
	\\
	=\delta^n \mE^{\beta^{\star},\mu_t^\star[h_t^c]}
	  \lpr W_{t+n}^{i,\beta^{i,\star}}\big(\underline{\Pi}_{t+n},X_{t+n}^i\big)
	- V^{i}\big(\underline{\Pi}_{t+n},X_{t+n}^i\big) \mid \underline{\mu}_t^\star[h_t^c],x_t^i \rpr.
	\end{multline}
Taking absolute value of both sides then using Jensen's inequality for $ f(x) = \vert x \vert $ and finally taking supremum over $ h_t^i $ reduces to
\begin{multline}
\sup_{h_t^i} \big\vert W_t^{i,\beta^{i,\star}}(\underline{\mu}_t^\star[h_t^c],x_t^i)  - V^{i}(\underline{\mu}_t^\star[h_t^c],x_t^i) \big\vert 
\\
\le \delta^n \sup_{h_t^i}  \mE^{\beta^{\star},\mu_t^\star[h_t^c]}
 \lpr\big\vert W_{t+n}^{i,\beta^{i,\star}}(\underline{\Pi}_{t+n},X_{t+n}^i)
 - V^{i}(\underline{\mu}_t^\star[h_t^c],x_t^i) \big\vert  \mid \underline{\mu}_t^\star[h_t^c],x_t^i \rpr.
\end{multline}
Now using the fact that $ W_{t+n},V^i $ are bounded and that we can choose $ n $ arbitrarily large, we get $ \sup_{h_t^i} \vert W_t^{i,\beta^{i,\star}}(\underline{\mu}_t^\star[h_t^c],x_t^i)  - V^{i}(\underline{\mu}_t^\star[h_t^c],x_t^i) \vert = 0 $.

\section{Intermediate Lemma used in Proof of Theorem~\ref{thih}} \label{app:interm_ih}

In this section, we present four lemmas. Lemma~\ref{thmfh1} and~\ref{lemcond} are intermediate technical results needed in the proof of Lemma~\ref{thmfh2}. Then the results in Lemma~\ref{thmfh2} and~\ref{lemfhtoih} are used in Appendix~\ref{app:ih} for the proof of Theorem~\ref{thih}. The proofs for Lemma~\ref{thmfh1} and \ref{lemcond} below aren't stated as they are analogous (the only difference being a non-zero terminal reward in the finite horizon model) to the proofs of Lemma~\ref{lemma:2} and \ref{lemma:3}, from Appendix~\ref{app:lemmas}, used in the proof of Theorem~\ref{Thm:Main}.

Define the reward-to-go $ W_t^{i,\beta^i,T} $ for any agent $ i $ and strategy $ \beta^i $  as
\begin{equation}\label{eqr2gfh}
W_t^{i,\beta^i,T}(h_t^i) = \mE^{\beta^i,\beta^{-i,\star},\mu_t^\star[h_t^c]} \big[ \sum_{n=t}^T \delta^{n-t} R^i(X_n,A_n)
+ \delta^{T+1-t} G^{i}(\underline{\Pi}_{T+1},X^i_{T+1}) \mid h_t^i \big].
\end{equation}
Here agent $ i $'s strategy is $ \beta^i $ whereas all other agents use strategy $ \beta^{-i,\star} $ defined above. Since $ \mX^i,\mA^i $ are assumed to be finite and $ G^i $ absolutely bounded, the reward-to-go is finite $ \forall $ $ i,t,\beta^i,h_t^i $.

In the following, any quantity with a $T$ in the superscript refers the finite horizon model with terminal reward $G^i$. For further discussion, please refer to the comments after the statement of Theorem~\ref{thih}.

\begin{lemma}\label{thmfh1}
	For any $ t \in \mathcal{T} $, $ i \in \mN $, $ h_t^i $ and $ \beta^i $,	
	\begin{equation} \label{eqintlem}
	V_t^{i,T}(\underline{\mu}_t^\star[h_t^c],x_t^i) \ge \mE^{\beta^i,\beta^{-i,\star},\mu_t^\star[h_t^c]} \big[ R^i(X_t,A_t)
	+ \delta V_{t+1}^{i,T}\big( \underline{F}(\underline{\mu}_t^{\star}[h_t^c],\beta_t^{\star}(\cdot|\underline{\mu}_t^*[h_t^c],\cdot),A_t) , X_{t+1}^i \big) \mid h_t^i \big].
	\end{equation}
\end{lemma}

\begin{lemma} \label{lemcond}
	\begin{multline}
	\mE^{\beta^i_{t+1:T},\beta_{t+1:T}^{-i,\star},\mu_{t+1}^\star[h_t^c,a_t]} \big[ \sum_{n=t+1}^T \delta^{n-(t+1)} R^i(X_n,A_n)
	+ \delta^{T+1-t} G^i(\underline{\Pi}_{T+1},X_{T+1}^i) \mid h_t^i,a_t,x_{t+1}^i \big]
	\\
	= \mE^{\beta^i_{t:T},\beta_{t:T}^{-i,\star},\mu_{t}^\star[h_t^c]} \big[ \sum_{n=t+1}^T \delta^{n-(t+1)} R^i(X_n,A_n)
	+ \delta^{T+1-t} G^i(\underline{\Pi}_{T+1},X_{T+1}^i) \mid h_t^i,a_t,x_{t+1}^i \big].
	\end{multline}
\end{lemma}
The result below shows that the value function from the backwards recursive algorithm is higher than any reward-to-go.

\begin{lemma}\label{thmfh2}
	For any $ t \in \mathcal{T} $, $ i \in \mN $, $ h_t^i $ and $ \beta^i $,	
	\begin{gather}
	V_t^{i,T}(\underline{\mu}_t^\star[h_t^c],x_t^i) \ge W_t^{i,\beta^i,T}(h_t^i).
	\end{gather}
\end{lemma}
\begin{IEEEproof}
We use backward induction for this. At time $ T $, using the maximization property from~\eqref{eq:m_FP} (modified with terminal reward $ G^i $),
\begin{subequations}
	\begin{align}
	&V_T^{i,T}(\underline{\mu}_T^\star[h_T^c],x_T^i)
	\\
	&\triangleq \mE^{\tilde{\gamma}_T^{i,T}(\cdot \mid x_T^i),\tilde{\gamma}_T^{-i,T},\mu_T^\star[h_t^c]} \big[ R^i(X_T,A_T) + \delta G^i\big( \underline{F}(\underline{\mu}_T^{\star}[h_T^c],\tilde{\gamma}_T^{T},A_T),X_{T+1}^i \big) \mid \underline{\mu}_T^\star[h_T^c],x_T^i \big]
	\\
	&\ge \mE^{{\gamma}_T^{i,T}(\cdot \mid x_T^i),\tilde{\gamma}_T^{-i,T},\mu_T^\star[h_t^c]} \big[ R^i(X_T,A_T)
	+ \delta G^i\big( \underline{F}(\underline{\mu}_T^{\star}[h_T^c],\tilde{\gamma}_T^{T},A_T) ,X_{T+1}^i \big) \mid \underline{\mu}_T^\star[h_T^c],x_T^i \big]
	\\
	&= W_T^{i,\beta^i,T}(h_T^i)
	\end{align}
\end{subequations}
Here the second inequality follows from~\eqref{eq:m_FP} and~\eqref{eq:Vdef} and the final equality is by definition in~\eqref{eqr2gfh}.

Assume that the result holds for all $ n \in \{t+1,\ldots,T\} $, then at time $ t $ we have
\begin{subequations}
	\begin{align}
	&V_t^{i,T}(\underline{\mu}_t^\star[h_t^c],x_t^i)
	\\
	&\ge \mE^{\beta_t^i,\beta_t^{-i,\star},\mu_t^\star[h_t^c]} \big[ R^i(X_t,A_t)
	+ \delta V_{t+1}^{i,T}\big( \underline{F}(\underline{\mu}_t^{\star}[h_t^c],\beta_t^{\star}(\cdot|\underline{\mu}_t^*[h_t^c],\cdot),A_t) , X_{t+1}^i \big) \mid h_t^i \big]
	\\
	&\ge \mE^{\beta_t^i,\beta_t^{-i,\star},\mu_t^\star[h_t^c]} \big[ R^i(X_t,A_t)
	+ \delta \mE^{\beta^i_{t+1:T},\beta_{t+1:T}^{-i,\star},\mu_{t+1}^\star[h_t^c,A_t]} \big[ \sum_{n=t+1}^T \delta^{n-(t+1)} R^i(X_n,A_n)
	\\ \nonumber
	&+ \delta^{T-t} G^i(\underline{\Pi}_{T+1},X_{T+1}^i) \mid h_t^i,A_t,X_{t+1}^i \big] \mid h_t^i \big]
	\\
	&= \mE^{\beta^i_{t:T},\beta^{-i,\star}_{t:T},\mu_t^\star[h_t^c]} \big[ \sum_{n=t}^T \delta^{n-t} R^i(X_n,A_n)
	+ \delta^{T+1-t}G^i(\underline{\Pi}_{T+1},X_{T+1}^i) \mid h_t^i \big]
	\\
	&= W_t^{i,\beta^i,T}(h_t^i)
	\end{align}
\end{subequations}
Here the first inequality follows from Lemma~\ref{thmfh1}, the second inequality from the induction hypothesis, the third equality follows from Lemma~\ref{lemcond} and the final equality by definition~\eqref{eqr2gfh}.
\end{IEEEproof}

The following result highlights the similarities between the fixed-point equation in infinite horizon and the backwards recursion in the finite horizon.

\begin{lemma}\label{lemfhtoih}
	Consider the finite horizon game with $ G^i \equiv V^i $. Then $ V_t^{i,T} = V^i $,  $ \forall $ $ i \in \mN $, $ t \in \{1,\ldots,T\} $ satisfies the backwards recursive construction stated above (adapted from \eqref{eq:m_FP} and \eqref{eq:Vdef}).

\end{lemma}	
\begin{IEEEproof}
	Use backward induction for this. Consider the finite horizon algorithm at time $ t=T $, noting that $ V_{T+1}^{i,T} \equiv G^i \equiv  V^i $,
	\begin{subequations} \label{eqfhT}
		\begin{align} 	
		\tilde{\gamma}_T^{i,T}(\cdot \mid x_T^i) &\in\!\!\!\!\! \argmax_{\gamma_T^i(\cdot \mid x_T^i) \in \Delta(\mA^i)} \!\!\! \mE^{\gamma_T^i(\cdot \mid x_T^i),\tilde{\gamma}_T^{-i,T},\upi_T^{-i}} \big[ R^i(X_T,A_T)
		+ \delta V^i\big( \underline{F}(\underline{\pi}_T, \tilde{\gamma}_T^{T}, A_T) , X_{T+1}^i \big) \mid \upi_T,x_T^i \big]
		\\
		V_T^{i,T}(\upi_T,x_T^i) &= \mE^{\tilde{\gamma}_T^{i,T}(\cdot \mid x_T^i),\tilde{\gamma}_T^{-i,T},\upi_T^{-i}} \big[ R^i(X_T,A_T)
		+ \delta V^i\big( \underline{F}(\underline{\pi}_T, \tilde{\gamma}_T^{T}, A_T) , X_{T+1}^i \big) \mid \upi_T,x_T^i \big].
		\end{align}
	\end{subequations}
	Comparing the above set of equations with~\eqref{eqihfpe}, we can see that the pair $ (V,\tilde{\gamma}) $ arising out of~\eqref{eqihfpe} satisfies the above. Now assume that $ V_n^{i,T} \equiv V^i $ for all $ n \in \{t+1,\ldots,T\} $. At time $ t $, in the finite horizon construction from~\eqref{eq:m_FP},~\eqref{eq:Vdef}, substituting $ V^i $ in place of $ V_{t+1}^{i,T} $ from the induction hypothesis, we get the same set of equations as~\eqref{eqfhT}. Thus $ V_t^{i,T} \equiv V^i $ satisfies it.
\end{IEEEproof}

\section{Proof of Theorem~\ref{thm:exist}} \label{app:existence}

	Denote the vector correspondence defined by the RHS of~\eqref{eqpp4} by
	\begin{gather}
	\phi(\underline{x}) = \begin{pmatrix}
	\phi_1(\underline{x}) \\
	\vdots \\
	\phi_4(\underline{x})
	\end{pmatrix} = \begin{pmatrix}
	\argmax_a a f_1(\underline{x}) \\
	\vdots \\
	\argmax_d d f_4(\underline{x})
	\end{pmatrix}
	\end{gather}
	where $ \underline{x} = (x,y,w,z) $. For any $ \underline{x} \in [0,1]^4 $, $ \phi(\underline{x}) $ is non-empty and closed, since the $ \argmax $ solution always exists and is one of $ \{0\},\{1\}, [0,1] $. If in addition $ \phi $ also has a \emph{closed graph} then by Kakutani Fixed Point Theorem there exists a solution to~\eqref{eqpp4}.
	
	Consider any sequence $ (\underline{x}_n,a_n,b_n,c_n,d_n) \rightarrow (\underline{x}_0,a_0,b_0,c_0,d_0) $ such that $ \forall $ $ n \ge 1 $,
	\begin{subequations} \label{eqr}
		\begin{gather}
		a_n \in \argmax_{a \in [0,1]} a f_1(\underline{x}_n), 
		\qquad 
		b_n \in \argmax_{b \in [0,1]} b f_2(\underline{x}_n),
		\\
		c_n \in \argmax_{c \in [0,1]} c f_3(\underline{x}_n),
		\qquad 
		d_n \in \argmax_{d \in [0,1]} d f_4(\underline{x}_n).	
		\end{gather}
	\end{subequations}
	We need to verify that~\eqref{eqr} also holds for the limit $ (\underline{x}_0,a_0,b_0,c_0,d_0) $.
	If $ \underline{x}_0 \notin \mathcal{D} $ then due to continuity, \eqref{eqr} indeed holds at the limit.
	For $ \underline{x}_0 \in \mathcal{D} $, for any $ i \in S(\underline{x}_0) $ if $ f_i(\underline{x}_0) = 0 $ then in the relation to be verified, the requirement is either of $ a_0,b_0,c_0,d_0 \in [0,1] $, which is always true.
	For $ \underline{x}_0 \in \mathcal{D}_1 \cap \mathcal{D}_2^\complement \cap \mathcal{D}_3^\complement \cap \mathcal{D}_4^\complement $, if $ f_1(\underline{x}_0) > 0 $ then for any sequence $ \underline{x}_n \rightarrow \underline{x}_0 $, for large $ n $ the points in the sequence are within $ B_\epsilon(\underline{x}_0) $ and thus $ f_1(\underline{x}_n) > 0 $ for large $ n $. This means that the relation from~\eqref{eqr} holds at the limit (noting that $ f_2,f_3,f_4 $ are continuous at $ \underline{x}_0 $ in this case).
	
	Similarly if $ f_1(\underline{x}_0) < 0 $ and for any $ \underline{x}_0 \in \mathcal{D}_1^\complement \cap \mathcal{D}_2 \cap \mathcal{D}_3^\complement \cap \mathcal{D}_4^\complement $.
	For $ \underline{x}_0 \in \mathcal{D}_1 \cap \mathcal{D}_2 \cap \mathcal{D}_3^\complement \cap \mathcal{D}_4^\complement $ if $ f_1(\underline{x}_0) > 0 $ and $ f_2(\underline{x}_0) < 0 $ then there exists an $ \epsilon > 0 $ such that $ \forall $ $ \underline{x} \in B_\epsilon(\underline{x}_0) $ we have $ f_1(\underline{x}) > 0 $ and $ f_2(\underline{x}) < 0 $. From this it follows that the relation~\eqref{eqr} holds at the limit. Similar argument works for any other sign combination of $ f_1,f_2,f_3,f_4 $.
	
	The two arguments above cover all cases.

%
\bibliographystyle{IEEEtran}

\begin{thebibliography}{10}
\providecommand{\url}[1]{#1}
\csname url@samestyle\endcsname
\providecommand{\newblock}{\relax}
\providecommand{\bibinfo}[2]{#2}
\providecommand{\BIBentrySTDinterwordspacing}{\spaceskip=0pt\relax}
\providecommand{\BIBentryALTinterwordstretchfactor}{4}
\providecommand{\BIBentryALTinterwordspacing}{\spaceskip=\fontdimen2\font plus
\BIBentryALTinterwordstretchfactor\fontdimen3\font minus
  \fontdimen4\font\relax}
\providecommand{\BIBforeignlanguage}[2]{{%
\expandafter\ifx\csname l@#1\endcsname\relax
\typeout{** WARNING: IEEEtran.bst: No hyphenation pattern has been}%
\typeout{** loaded for the language `#1'. Using the pattern for}%
\typeout{** the default language instead.}%
\else
\language=\csname l@#1\endcsname
\fi
#2}}
\providecommand{\BIBdecl}{\relax}
\BIBdecl

\bibitem{Sh53}
L.~S. Shapley, ``Stochastic games,'' \emph{Proceedings of the national academy
  of sciences}, vol.~39, no.~10, pp. 1095--1100, 1953.

\bibitem{BaOl98}
T.~Ba{\c s}�ar and G.~Olsder, \emph{Dynamic Noncooperative Game Theory, 2nd
  Edition}.\hskip 1em plus 0.5em minus 0.4em\relax Society for Industrial and
  Applied Mathematics, 1998.

\bibitem{FiVr12}
J.~Filar and K.~Vrieze, \emph{Competitive Markov decision processes}.\hskip 1em
  plus 0.5em minus 0.4em\relax Springer Science \& Business Media, 2012.

\bibitem{OsRu94}
M.~J. Osborne and A.~Rubinstein, \emph{A Course in Game Theory}, ser. MIT Press
  Books.\hskip 1em plus 0.5em minus 0.4em\relax The MIT Press, 1994, vol.~1.

\bibitem{FuTi91book}
D.~Fudenberg and J.~Tirole, \emph{Game Theory}.\hskip 1em plus 0.5em minus
  0.4em\relax Cambridge, MA: MIT Press, 1991.

\bibitem{samuelson2006}
G.~J. Mailath and L.~Samuelson, \emph{Repeated games and reputations: long-run
  relationships}.\hskip 1em plus 0.5em minus 0.4em\relax Oxford university
  press, 2006.

\bibitem{MaTi01}
E.~Maskin and J.~Tirole, ``Markov perfect equilibrium: I. observable actions,''
  \emph{Journal of Economic Theory}, vol. 100, no.~2, pp. 191--219, 2001.

\bibitem{ErPa95}
R.~Ericson and A.~Pakes, ``Markov-perfect industry dynamics: A framework for
  empirical work,'' \emph{The Review of Economic Studies}, vol.~62, no.~1, pp.
  53--82, 1995.

\bibitem{BeVa96}
D.~Bergemann and J.~V{\"a}lim{\"a}ki, ``Learning and strategic pricing,''
  \emph{Econometrica: Journal of the Econometric Society}, pp. 1125--1149,
  1996.

\bibitem{AcRo01}
D.~Acemo\u{g}lu and J.~A. Robinson, ``A theory of political transitions,''
  \emph{American Economic Review}, pp. 938--963, 2001.

\bibitem{KrSo94}
D.~M. Kreps and J.~Sobel, ``Chapter 25 signalling,'' ser. Handbook of Game
  Theory with Economic Applications.\hskip 1em plus 0.5em minus 0.4em\relax
  Elsevier, 1994, vol.~2, pp. 849 -- 867.

\bibitem{DoPa07}
U.~Doraszelski and A.~Pakes, ``A framework for applied dynamic analysis in
  {IO},'' \emph{Handbook of industrial organization}, vol.~3, pp. 1887--1966,
  2007.

\bibitem{AlKaSi09}
E.~Altman, V.~Kambley, and A.~Silva, ``Stochastic games with one step delay
  sharing information pattern with application to power control,'' in
  \emph{Game Theory for Networks, 2009. GameNets' 09. International Conference
  on}.\hskip 1em plus 0.5em minus 0.4em\relax IEEE, 2009, pp. 124--129.

\bibitem{NaGuLaBa14}
A.~Nayyar, A.~Gupta, C.~Langbort, and T.~Ba{\c s}ar, ``Common information based
  {M}arkov perfect equilibria for stochastic games with asymmetric information:
  Finite games,'' \emph{IEEE Trans.~Automatic Control}, vol.~59, no.~3, pp.
  555--570, March 2014.

\bibitem{Sp73}
M.~Spence, ``Job market signaling,'' \emph{The quarterly journal of Economics},
  pp. 355--374, 1973.

\bibitem{Gr81}
S.~J. Grossman, ``The informational role of warranties and private disclosure
  about product quality,'' \emph{The Journal of Law \& Economics}, vol.~24,
  no.~3, pp. 461--483, 1981.

\bibitem{Wi77}
C.~Wilson, ``A model of insurance markets with incomplete information,''
  \emph{Journal of Economic theory}, vol.~16, no.~2, pp. 167--207, 1977.

\bibitem{RoSt76}
M.~Rothschild and J.~Stiglitz, ``Equilibrium in competitive insurance markets:
  An essay on the economics of imperfect information,'' in \emph{Foundations of
  Insurance Economics}.\hskip 1em plus 0.5em minus 0.4em\relax Springer, 1976,
  pp. 355--375.

\bibitem{Za75}
A.~Zahavi, ``Mate selection--a selection for a handicap,'' \emph{Journal of
  theoretical Biology}, vol.~53, no.~1, pp. 205--214, 1975.

\bibitem{Ba92}
A.~V. Banerjee, ``A simple model of herd behavior,'' \emph{The Quarterly
  Journal of Economics}, pp. 797--817, 1992.

\bibitem{BiHiWe92}
\BIBentryALTinterwordspacing
S.~Bikhchandani, D.~Hirshleifer, and I.~Welch, ``\BIBforeignlanguage{English}{A
  theory of fads, fashion, custom, and cultural change as informational
  cascades},'' \emph{\BIBforeignlanguage{English}{Journal of Political
  Economy}}, vol. 100, no.~5, pp. pp. 992--1026, 1992. [Online]. Available:
  \url{http://www.jstor.org/stable/2138632}
\BIBentrySTDinterwordspacing

\bibitem{SmSo02}
\BIBentryALTinterwordspacing
L.~Smith and P.~S\"orensen, ``Pathological outcomes of observational
  learning,'' \emph{Econometrica}, vol.~68, no.~2, pp. 371--398, 2000.
  [Online]. Available: \url{http://dx.doi.org/10.1111/1468-0262.00113}
\BIBentrySTDinterwordspacing

\bibitem{DePeSi15}
N.~R. Devanur, Y.~Peres, and B.~Sivan, ``Perfect {B}ayesian equilibria in
  repeated sales,'' in \emph{Proceedings of the Twenty-Sixth Annual ACM-SIAM
  Symposium on Discrete Algorithms}.\hskip 1em plus 0.5em minus 0.4em\relax
  SIAM, 2015, pp. 983--1002.

\bibitem{Ho80}
Y.-C. Ho, ``Team decision theory and information structures,''
  \emph{Proceedings of the IEEE}, vol.~68, no.~6, pp. 644--654, 1980.

\bibitem{NaMaTe11}
A.~Nayyar, A.~Mahajan, and D.~Teneketzis, ``Optimal control strategies in
  delayed sharing information structures,'' \emph{IEEE Trans.~Automatic
  Control}, vol.~56, no.~7, pp. 1606--1620, July 2011.

\bibitem{GuNaLaBa14}
A.~Gupta, A.~Nayyar, C.~Langbort, and T.~Ba{\c s}ar, ``Common information based
  {M}arkov perfect equilibria for linear-gaussian games with asymmetric
  information,'' \emph{SIAM Journal on Control and Optimization}, vol.~52,
  no.~5, pp. 3228--3260, 2014.

\bibitem{OuTaTe15}
Y.~Ouyang, H.~Tavafoghi, and D.~Teneketzis, ``Dynamic oligopoly games with
  private {M}arkovian dynamics,'' in \emph{Proc. 54th IEEE Conf. Decision and
  Control (CDC)}, 2015.

\bibitem{shamma14}
L.~Li and J.~Shamma, ``Lp formulation of asymmetric zero-sum stochastic
  games,'' in \emph{53rd IEEE Conference on Decision and Control}, Dec 2014,
  pp. 1930--1935.

\bibitem{cole2001}
H.~L. Cole and N.~Kocherlakota, ``Dynamic games with hidden actions and hidden
  states,'' \emph{Journal of Economic Theory}, vol.~98, no.~1, pp. 114--126,
  2001.

\bibitem{nayyar2013}
A.~Nayyar, A.~Mahajan, and D.~Teneketzis, ``Decentralized stochastic control
  with partial history sharing: A common information approach,'' \emph{IEEE
  Transactions on Automatic Control}, vol.~58, no.~7, pp. 1644--1658, July
  2013.

\bibitem{Ma13}
A.~Mahajan, ``Optimal decentralized control of coupled subsystems with control
  sharing,'' \emph{Automatic Control, IEEE Transactions on}, vol.~58, no.~9,
  pp. 2377--2382, 2013.

\bibitem{MaTe08}
A.~Mahajan and D.~Teneketzis, ``On the design of globally optimal communication
  strategies for real-time communcation systems with noisy feedback,''
  \emph{IEEE J.~Select.~Areas Commun.}, no.~4, pp. 580--595, May 2008.

\bibitem{Nate08}
A.~Nayyar and D.~Teneketzis, ``On globally optimal real-time encoding and
  decoding strategies in multi-terminal communication systems,'' in
  \emph{Proc.~IEEE Conf. on Decision and Control}, Cancun, Mexico, Dec. 2008,
  pp. 1620--1627.

\bibitem{VaAn14}
D.~Vasal and A.~Anastasopoulos, ``Stochastic control of relay channels with
  cooperative and strategic users,'' \emph{IEEE Transactions on
  Communications}, vol.~62, no.~10, pp. 3434--3446, Oct 2014.

\bibitem{chokreps87}
I.-K. Cho and D.~M. Kreps, ``Signaling games and stable equilibria,'' \emph{The
  Quarterly Journal of Economics}, vol. 102, no.~2, pp. 179--221, 1987.

\bibitem{cho87}
I.-K. Cho, ``A refinement of sequential equilibria,'' \emph{Econometrica},
  vol.~55, no.~6, pp. 1367--1389, 1987.

\bibitem{sobel87}
J.~S. Banks and J.~Sobel, ``Equilibrium selection in signaling games,''
  \emph{Econometrica}, vol.~55, no.~3, pp. 647--661, 1987.

\bibitem{Na51}
J.~Nash, ``Non-cooperative games,'' \emph{Annals of mathematics}, pp. 286--295,
  1951.

\bibitem{VaAn16a}
D.~Vasal and A.~Anastasopoulos, ``Signaling equilibria for dynamic {LQG} games
  with asymmetric information,'' in \emph{Proc.~IEEE Conf. on Decision and
  Control}, Dec. 2016, pp. 6901--6908.

\bibitem{VaAn16b}
------, ``Decentralized {B}ayesian learning in dynamic games,'' in
  \emph{Proc.~Allerton Conf.~Commun., Control, Comp.}, Sept. 2016.

\end{thebibliography}

\end{document}